\documentclass[11pt]{article}
\usepackage{amsmath}
\usepackage{amssymb}
\usepackage{graphicx}
\usepackage{bbm}
\usepackage{url}
\usepackage{amsfonts}
\usepackage{txfonts}
\usepackage{mathrsfs}
\usepackage{subfigure}
\usepackage{verbatim}
\usepackage{pmat}
\usepackage{multirow}
\usepackage{listings}
\usepackage{algorithm}
\usepackage{algpseudocode}

\usepackage{color}

\newtheorem{thm}{\bf Theorem}[section]

\newtheorem{lem}{\bf Lemma}[section]
\newtheorem{exam}{\bf Example}[section]
\newenvironment{proof}{{\noindent\it Proof}\quad}{\hfill $\square$\par}
\newtheorem{rem}{\bf Remark}

\numberwithin{equation}{section}

\title{Randomized Core Reduction for Discrete Ill-Posed Problem}

\author{
Liping Zhang
\thanks{Department of Mathematics, Zhejiang University of Technology, Hangzhou 310023, PR China. (zhanglp@zjut.edu.cn).}
\and
Yimin Wei
\thanks{Corresponding author (Y. Wei). ( ymwei@fudan.edu.cn and yimin.wei@gmail.com). School of Mathematical Sciences and Key Laboratory of Mathematics for Nonlinear Sciences, Fudan University, Shanghai 200433, PR China.  }
}

\begin{document}


\maketitle
\begin{abstract}

In this paper, we apply randomized algorithms to approximate the total least squares (TLS) solution of the problem $Ax\approx b$ in the large-scale discrete ill-posed problems. A regularization technique, based on the multiplicative randomization and the subspace iteration, is proposed to obtain the approximate core problem.  In the error analysis, we provide upper bounds 
for the errors of the solution and the residual of the randomized core reduction. Illustrative numerical examples and comparisons are presented.
\end{abstract}



\section{Introduction}\label{sec_int}

Consider the discrete ill-posed linear system
\begin{equation}\label{eq:Axb}
Ax\approx b, \quad\quad A\in \mathbb{R}^{m\times n},(m\geq n)
\end{equation}
where the matrix $A$ is of  full column rank and numerical low-rank.
In practice,
many discrete ill-posed problems arising from physics and engineering can be reduced to the problem \eqref{eq:Axb}.
To reduce the severe instability of \eqref{eq:Axb}, we introduce the approximate core problem, which is  well-conditioned, low dimensional and can be obtained by randomized algorithms.

The concept of core problem is proposed by Paige and Strako\v{s} in \cite{paige2006core} and
 used to find the minimum norm solution of the TLS problem.
In detail, for the matrix $A\in \mathbb{R}^{m\times n}$, $b\in \mathbb{R}^{m}$, there exist orthogonal matrices $P^{\mathrm T} = P^{-1}$ and $Q^{\mathrm T} = Q^{-1}$, satisfying
\begin{equation}\label{core_s}
P^{\mathrm T} \begin{bmatrix} b & A\end{bmatrix}\begin{bmatrix} 1 & \\ & Q \end{bmatrix}
= \begin{bmatrix}b_1 &A_{11}&  \bf{0} \\ \bf{0}&\bf{0} &A_{22}\end{bmatrix},
\end{equation}
where $b_1\in \mathbb{R}^r$ and $A_{11}$ are of minimal dimensions.
The sub-problem defined by $[A_{11},~b_1]$, leading to the sufficient and necessary conditions for solving the original problem $Ax\approx b$, is called the \emph{core problem}.
The remaining part $A_{22}$ has a trivial (zero) right-hand side and a maximal dimension.
This transformation can be obtained by the singular value decomposition (SVD) of $A$, the Householder transformation \cite{hnetynkova2013core,paige2006core} and the Golub-Kahan bidiagonalization; see \cite{hnetynkova2007lanczos}.

An important application of the core problem is the TLS problem, which considers the perturbations of the coefficient matrix $A$ and the right-hand side $b$ simultanously, i.e.,
\begin{equation}\label{TLS}
    \min_{E,f}\|[E, ~f]\|_F, \quad \textrm{subject to } (A+E)x = b+f.
\end{equation}
If $\sigma_n(A) > \sigma_{n+1}([b,~A])$, the above TLS problem has the closed-form \cite[Theorem 2.7]{van1991total}
\begin{equation}\label{TLS_close}
  (A^{\mathrm T} A -\sigma_{n+1}^2([b,~A])I_n) x = A^{\mathrm T} b.
\end{equation}
Substituting the decomposition of the core problem \eqref{core_s},
we can get the solution of  \eqref{TLS} by  solving the  following TLS problem of lower dimension,
\begin{equation}\label{TLS_1}
 \min_{E_1,f_1}\|[E_1, ~f_1]\|_F, \quad \textrm{subject to } (A_{11}+E_1)y = b_1+f_1,
\end{equation}
or the corresponding closed-form,
\begin{equation}\label{close_1}
  \left[A_{11}^{\mathrm T} A_{11} - \sigma_{r+1}^2([b_1,~A_{11}])I_r\right]y = A_{11}^{\mathrm T} b_1,
\end{equation}
and by back-transformation of $y$ to the original coordinates, $x = Q[y^{\mathrm T},~0]^{\mathrm T}$.

The classical method for solving the small-scale TLS problem \eqref{TLS_1} is based on the SVD of augmented matrix $[A_{11},~b_1]$; see \cite[Section 2.3.2]{van1991total}.
There also exist some other efficient methods, such as
Lanczos or Golub-Kahan bidiagonalization \cite{lampe2010solving} and the Rayleigh quotient iteration \cite{bjorck2000methods}.
If we have the SVD of the coefficient matrix $A_{11}$ in advance, the exact solution of \eqref{TLS_1} can also be expressed by the the SVD of $A_{11}$ based on the closed-form \eqref{close_1}; see Lemma \ref{lem_svdA} for detail.
The truncated TLS is an effective regularization method for solving ill-posed problems \cite{fierro1997regularization}. With the SVD of augmented matrix $[A,~b]$,
\begin{equation*}
\begin{bmatrix}
  A & b
\end{bmatrix} = \sum_{i=1}^{n+1}\bar{u}_i\bar{\sigma}_i\bar{v}_i^{\mathrm T} =\bar{U}\bar{\Sigma} \bar{V}^{\mathrm T} = \bar{U}\bar{\Sigma}
\begin{bmatrix} \bar{V}_{11}&\bar{V}_{12}\\ \bar{V}_{21} & \bar{V}_{22}\end{bmatrix}^{\mathrm T},
\end{equation*}
choose a truncation parameter $t \leq  \min\{n, {\rm rank}[A, ~b]\}$ such that
\begin{equation*}
  \bar{\sigma}_{t+1}<\bar{\sigma}_t, \quad {\rm and}\quad 0\neq \bar{V}_{22}\in \mathbb{R}^{1\times (n-t+1)}.
\end{equation*}
It is reasonable to assume a well-defined gap in the singular value spectrum, though generally, rank determination is a difficult problem, even with the SVD \cite{fierro1994collinearity}.
   Denoting $[\tilde{A},~\tilde{b}] = \sum_{i=1}^{t}\bar{u}_i\bar{\sigma}_i\bar{v}_i^{\mathrm T}$, the truncated TLS solution $\tilde{x}$ is the minimum norm solution to $\tilde{A}x=\tilde{b}$ and the minimum norm LS solution to $\min_x\|\tilde{A}x-b\|$; see \cite{fierro1994collinearity}. Consequently, we obtain
 \begin{equation}\label{x_trun}
   \tilde{x} = -\bar{V}_{12}\bar{V}_{22}^\dag = (\bar{V}_{11}^{\mathrm T})^\dag (\bar{V}_{21}^{\mathrm T})=\tilde{A}^\dag b.
 \end{equation}
 When $t=n$, the solution $\tilde{x} = x$ gives the exact solution of the TLS problem \eqref{TLS}.

The core concept to the case with multiple right-hand sides $AX\approx B$ is considered  in \cite{hnetynkova2013core, hnetynkova2016solvability, hnetynkova2011total} and realized by the Golub-Kahan bidiagonalization \cite{hnetynkova2015band}.
Recently Hn\v{e}tynkov\'{a} \emph{et al.} extend the core reduction to tensor for problems with structured right-hand sides \cite{hnetynkova2018tls}.
 Note that different kinds of condition number of the multidimensional TLS has been given in \cite{zheng2017condition} and the TLS minimization with multiple right-hand sides with respect to different unitarily invariant norms is considered in \cite{wang2017total}.

With the core problem, the dimension of the original problem is reduced. However, applying the classical tools such as the SVD to obtain the exact core problem is unrealistic for large-scale problems.
Furthermore, for the ill-posed problems that the larger singular values dominate the solution, the full SVD seems unnecessary.
Thus we propose the approximate core problem by randomized algorithms, to get ride of decompositions of large matrices, which can be regards as a regularization technique.

Recently, different kinds of randomized algorithms have been proposed to compute  the low-rank matrix approximation \cite{ avron2010blendenpik, coakley2011fast, drineas2011faster, gower2015randomized, halko2011finding, mahoney2010randomized, martinsson2011randomized, meng2011lsrn, rachkovskij2012randomized,rokhlin2008fast, teng2018fast, woodruff2014sketching, woolfe2008fast}.
The main idea is to obtain a projection by a random matrix (Gaussian matrix or matrix generated by the sub-sampled randomized Fourier transform (SRFT) \cite{meng2011lsrn, rokhlin2008fast, woolfe2008fast}) or random sampling \cite{avron2010blendenpik, martinsson2011randomized} with preconditioning \cite{coakley2011fast, rachkovskij2012randomized}; refer also to the review paper \cite{halko2011finding}. Gu presented a randomized algorithm within the subspace iteration framework which gives accurate low-rank approximations of high probability, for matrices with rapidly decaying singular values \cite{gu2015subspace}.
 By the randomized algorithms proposed in \cite{halko2011finding}, the ill-posed problems are solved efficiently by Xiang and Zou in \cite{xiang2013regularization, xiang2014randomized}. We also provide the error analysis for the randomized generalized singular value decomposition (GSVD) in \cite{wei2016tikhonov}.
 Jia and Yang improve our bounds of the approximation accuracy  for severely, moderately and mildly ill-posed problems; see \cite{jia2018modified}.
Randomized algorithms are also used for the generalized
Hermitian eigenvalue problems by Saibaba \emph{et al.} \cite{saibaba2015randomized}
and for the TLS problem by Xie \emph{et al.}  \cite{xie2014statistical}.

In this paper, we propose a randomized algorithm based on the subspace iteration method for the linear system \eqref{eq:Axb}, and construct an approximate core problem for this system. If $A$ has a significantly low numerical rank, the dimension of the approximate core problem can be much reduced. We can prove that the smaller system yields an accurate approximate solution.

The paper is organized as follows. 
Randomized algorithms are proposed in Section \ref{sec_randomized}, with the error analyses in Section \ref{sec_error}.
The improvement in time and memory requirements are
illustrated with numerical examples in Section \ref{sec_numerical} and
Section \ref{sec conclusion} concludes the paper.

 Throughout this paper, $\mathbb{R}^{m\times n}$
denotes the set of $m\times n$ matrices with real entries. $I_n$ stands for the
identity matrix of order $n$.
All the norm $\|\cdot\|$  is the 2-norm.
The $k$-th singular value of $A$ is $\sigma_{k}(A)$ and
$\sigma'_k=\sigma_k(A)$, $\bar{\sigma}_k = \sigma_k([A,~b])$.
Denote the approximate matrix of $A$ by a randomized algorithm as $A_r$,  $\mathcal{R}(A)$ is the range space of $A$ and
${\rm vec}(A)$ is vectorization of matrix $A$.
Moreover, a standard Gaussian random matrix has independent standard normal components.

\section{Randomized Algorithm}\label{sec_randomized}
A large amount of research has considered randomized algorithms recently \cite{avron2010blendenpik,coakley2011fast,halko2011finding,
rachkovskij2012randomized,rokhlin2008fast,saibaba2015randomized,
xiang2014randomized,xie2014statistical}. Well-designed randomized algorithms
are potentially more efficient, especially for large-scale problems.

For the original problem $Ax\approx b$, we can derive the approximate core problem as follows,
\begin{equation}\label{core_appx}
 P^{\mathrm T} \begin{bmatrix}b & A\end{bmatrix}\begin{bmatrix} 1 & \\ & Q \end{bmatrix}
 = \begin{bmatrix}b_1 & A_{11}& A_{12} \\ b_2 & A_{21} &A_{22}\end{bmatrix}
 \approx \begin{bmatrix}b_1 & A_{11}&  \bf{0} \\ \bf{0} & \bf{0} &A_{22}\end{bmatrix},
\end{equation}
where $\max\{\|b_2\|,~\|A_{12}\|,~\|A_{21}\|\} \leq \delta$ for some small $\delta>0$.
Then the original problem $Ax \approx b$ can be solved approximately by $A_{11}y\approx b_1$ with $x\approx Q [ y^{\mathrm T} , 0 ]^{\mathrm T}$. The dimension of the problem is reduced evidently.

Now we adopt randomized algorithms to achieve the approximate core problem in \eqref{core_appx}.
Since the ill-poseness stems from the coefficient matrix $A$, we project $A$ to a small subspace of $\mathcal{R}(A)$ with $b$ projected accordingly.
Then a small approximate core problem is obtained.
This randomization idea has been used on the SVD \cite{xiang2013regularization} and the GSVD for regularization \cite{wei2016tikhonov}.


First we cite an important inequality  for randomized algorithms.

\begin{lem}\emph{\cite[Corollary 10.9]{halko2011finding}}\label{lem_Qconv}
Suppose that $A\in \mathbb{R}^{m\times n}$ has the singular values $\sigma_1'\geq \sigma_2'
\geq \ldots \geq \sigma_n'$. Let $\Omega$ be an $n\times (k+s)$ standard Gaussian matrix with $k+s\leq \min\{m,n\}$, $s \geq 4$, and $Q$ be an orthonormal basis for the range of the sampled matrix $A \Omega$. Then
\begin{equation}\label{est_Q}
\left\|A-QQ^{\mathrm T} A\right\|\leq \left(1+16\sqrt{1+\frac{k}{s+1}}\right)\sigma'_{k+1}+\frac{8\sqrt{k+s}}{s+1}\sqrt{
\sum_{j>k}{\sigma'}_j^2}
\end{equation}
with probability not less than $1-3e^{-s}$.
\end{lem}

Gu gave a stronger result if $Q$ is selected by subspace iteration and the large deviation bound is given as follows.

\begin{lem}\emph{\cite[Theorem 5.8]{gu2015subspace}}\label{lem_Qgu}
Suppose that $A\in \mathbb{R}^{m\times n}$ has the singular values $\sigma_1'\geq \sigma_2'
\geq \ldots \geq \sigma_n'$. Let $\Omega$ be an $n\times (k+s)$ standard Gaussian matrix with $k+s\leq \min\{m,n\}$, $s \geq p\geq 0$, and $Q$ be an orthonormal basis for the range of the sampled matrix $(AA^{\mathrm T})^qA \Omega$.
Given any $0<\Delta\ll 1$, define
\begin{equation}\label{c_delta}
  \mathcal{C}_\Delta = \frac{e\sqrt{k+s}}{p+1}\left(\frac{2}{\Delta}\right)^{\frac{1}{p+1}}
  \left(\sqrt{n-k-s+p}+\sqrt{k+s}+\sqrt{2\log\frac{2}{\Delta}}\right).
\end{equation}
We then have
\begin{equation}\label{est_Qnew}
\left\|A-QQ^{\mathrm T} A\right\|\leq \sqrt{{\sigma'}_{k+1}^2+k\mathcal{C}_\Delta^2{\sigma'}_{k+1+s-p}^2
\left(\frac{{\sigma'}_{k+1+s-p}}{{\sigma'}_k}\right)^{4q}}\leq \sqrt{1+k\mathcal{C}_\Delta^2
\left(\frac{{\sigma'}_{k+1+s-p}}{{\sigma'}_k}\right)^{4q}}{\sigma'}_{k+1},
\end{equation}
with probability not less than $1-\Delta$.
\end{lem}

Here $s$ is the over-sampling parameter, and its selection is crucial for the effectiveness of the randomized algorithms.
A small number of columns  are added  to provide  flexibility \cite{halko2011finding}.
The additional parameter $p$ is to balance the need for oversampling for reliability and faster convergence \cite{gu2015subspace}.
In practiceï¼? the orthonormal matrix $Q$ can be selected by adaptive algorithm \cite[Algorithm 4.2]{halko2011finding} combined with the subspace iteration \cite[Algorithm 4.4]{halko2011finding}, as in Algorithm~\ref{algo_tls1} below.
All the operations in the algorithm are implemented in
a flexible fashion that allows the matrix $A$ to be available as a (sparse) matrix or a function handle. We only need the matrix-vector products with $A$ and $A^{\rm T}$, usually efficient for large-scale problems.

\begin{algorithm}[ht]
\caption{Randomized subspace iteration with adaptive range finder}\label{algo_tls1}
\begin{algorithmic}[1]
\Require
 $A\in \mathbb{R}^{m\times n}$, a tolerance $\varepsilon$, an integer $\ell$ (e.g., $\ell = 10$).
\Ensure
Orthonormal matrix $Q$.
\State Draw standard Gaussian vectors $\omega^{(1)},\ldots,\omega^{(\ell)}$ of length $n$.
\State For $i = 1, 2,\ldots, \ell$, compute $y^{(i)} = A\omega^{(i)}$.
\State $j= 0$. $Q^{(0)} = [ ]$, the $m\times0$ empty matrix.
\While {$\max\left\{\left\|y^{(j+1)}\right\|,\left\|y^{(j+2)}\right\|,\ldots, \left\|y^{(j+\ell)}\right\|\right\}>\varepsilon/(10\sqrt{2/\pi})$,}
\State j=j+1.
\State Overwrite $y^{(j)}$ by $(I-Q^{(j-1)}Q^{(j-1){\mathrm T}})y^{(j)}$.
\State $q^{(j)} =y^{(j)}/\|y^{(j)}\|$, $Q^{(j)} = [Q^{(j-1)}, q^{(j)}]$.
\State Draw a standard Gaussian vector $\omega^{(j+\ell)}$ of length $n$.
\State $y^{(j+\ell)} = (I-Q^{(j)}Q^{(j){\mathrm T}})A\omega^{(j+\ell)}$
\For {$i = (j+1), (j+2), \ldots, (j+\ell-1)$,}
\State Overwrite $y^{(i)}$ by $y^{(i)}- q^{(j)}\langle q^{(j)},y^{(i)} \rangle$.
\EndFor
\EndWhile
\State $Q_0 = Q^{(j)}$.
\For {$j=1,2,\ldots,q$}
\State Form $\widetilde{Y}_j = A^{\rm T}Q_{j-1}$ and compute its QR factorization $\widetilde{Y}_j = \widetilde{Q}_j \widetilde{R}_j$.
\State Form $Y_j = A \widetilde{Q}_j$ and compute its QR factorization $Y_j = Q_jR_j$.
\EndFor
\State $Q=Q_r$.
\end{algorithmic}
\end{algorithm}

\subsection{Randomized Core Reduction}
A randomized SVD of $A$ has been given in \cite{halko2011finding}, that is,
$A\approx U_1 \Sigma_1 V_1^{\mathrm T}$, where $U_1\in \mathbb{R}^{m\times r}$, $V_1\in \mathbb{R}^{n\times r}$ are column orthogonal, $\Sigma_1 = {\rm diag}(\sigma_1 I,\ldots,\sigma_t I)$,
the singular values $\sigma_j$ has the multiplicity of $h_j$, and $\sum_{j=1}^t h_j = r$.
Define 
$\Sigma={\rm diag}(\Sigma_1, \bf{0})$
and expand $U_1$ and $V_1$ to orthogonal matrices
$U = [U_1,~ U_2]$ and $V = [V_1,~V_2]$.
Hence, $A \approx U_1\Sigma_1 V_1^{\mathrm T}= U\Sigma V^{\mathrm T}$.
We will follow the process in \cite[Section 2]{paige2006core} for the approximate core problem,
ultimately constrained to have simple singular values.
For the argumented matrix $[b, ~A]$, we have
\begin{eqnarray*}
\quad U^{\mathrm T}\begin{bmatrix}b & A\end{bmatrix}\begin{bmatrix}1 & \\ & V\end{bmatrix}
= \begin{bmatrix}U^{\mathrm T} b & U^{\mathrm T} A V\end{bmatrix}\approx
\begin{bmatrix}U^{\mathrm T} b & \Sigma\end{bmatrix}
= \begin{bmatrix}U_1^{\mathrm T} b & \Sigma_1 & \\ U_2^{\mathrm T} b & & \bf{0}\end{bmatrix}.
\end{eqnarray*}
Assume that $\sigma_1>\sigma_2>\ldots>\sigma_t$ ($r\geq t$), i.e., $\Sigma_1\in \mathbb{R}^{r\times r}$ contains $t$ distinct singular values, and the corresponding partition
$U_1 = [u_1,~ u_2 ,\cdots , u_t ]$.
Please note that $u_j$ is a submatrix of size $m\times h_j$, where $h_j$ is the multiplicity of $\sigma_j$ ($j=1,2,\ldots,t$).
For $i=1,2,\cdots,t$, choose a sequence of Householder transformation $S_i$ such that $S_i u_i^{\mathrm T} b = [\varphi_i, 0,\ldots,0]^{\rm T}$ and $S_{t+1}$ such that $S_{t+1} U_2^{\mathrm T} b = [\varphi_{t+1},0,\dots,0]^{\rm T}$, where $\varphi_i = \|u_i^{\mathrm T} b\|$, $\varphi_{t+1}=\|U_2^{\mathrm T} b\|$. 
It is not necessary to generate $U_2^{\mathrm T} b$ and $S_{t+1}$ exactly, and
$\varphi_{t+1} = \|U_2^{\mathrm T} b\| = \left\|b - U_1U_1^{\mathrm T} b\right\|$.
There exists a permutation matrix $\Pi$ that moves the zero elements of
$[\varphi_1,0,\ldots,0, \ldots,\varphi_t,0,\dots,0,\varphi_{t+1},0,\ldots,0]^{\mathrm T}$ to the bottom of this
vector, leaving the nonzero $[\varphi_1,\ldots, \varphi_t,\varphi_{t+1}]^{\rm T}$ at the top while keeping
the $t\times t$ sub-matrix  $\Sigma_1$ diagonal.
With the orthogonal matrix $S = {\rm diag}(S_1 ,\ldots,S_t, S_{t+1})$,
we produce
\begin{equation*}
\Pi S \begin{bmatrix}U_1^{\mathrm T} b & \Sigma_1 & \\ U_2^{\mathrm T} b & & \bf{0}\end{bmatrix}
\begin{bmatrix} 1 & \\ & S\Pi\end{bmatrix} =
\begin{bmatrix}\varphi_1  & \sigma_1 &&&& \\ \vdots && \ddots &&&\\ \varphi_t &&& \sigma_t && \\ \varphi_{t+1} &&&& 0& \\ \bf{0} &&&& & \Sigma_2 \end{bmatrix}.
\end{equation*}
For the special case with the multiplicity $h_j=1$ ($j=1,2,\ldots, t$), we have $S_j=1$, $\Pi = I$ and $\Sigma_2=\bf{0}$.
Here we assume that $\varphi_i \neq 0$, $i = 1,2,\ldots, t$, as we permute all zeros to the bottom.

Consequently, we have
\begin{equation}\label{svd_A}
\quad\Pi S  U^{\mathrm T}\begin{bmatrix}b & A\end{bmatrix}\begin{bmatrix}1 & \\ & VS\Pi^{\mathrm T}\end{bmatrix}
\approx
\begin{pmat}[{...|.}]
\varphi_1 &\sigma_1&&& &\cr
\vdots &  &\ddots&&  &\cr
\varphi_t  && &\sigma_t&&  \cr
\varphi_{t+1}&&&  0&  &\cr\-
\bf{0}&&&&  \bf{0} &\Sigma_2 \cr
\end{pmat}\triangleq
\begin{pmat}[{.|.}]
b_1 &A_{11}&\bf{0}\cr\-
\bf{0}&  \bf{0} &A_{22} \cr
\end{pmat},
\end{equation}
in the form of \eqref{core_appx} with $P= U S \Pi^{\mathrm T}$ and $Q=VS\Pi^{\mathrm T}$.
The approximate core problem is given by $\{b_1, ~A_{11}\}$, where $b_1\in \mathbb{R}^{t+1}$ and $A_{11} \in \mathbb{R}^{(t+1)\times t}$.
Furthermore, if $\varphi_{t+1} = 0$, then the approximate core problem of TLS problem degenerates to a linear equation with $b_1\in \mathbb{R}^{t}$ and $A_{11} \in \mathbb{R}^{t\times t}$.
The  construction of the permutation matrix $\Pi$ is given by the permutation:
\begin{equation*}
\begin{pmatrix}2 & 3 &  \ldots & t & t+1 \\
h_1 +1 & h_1 +h_2+1 & \ldots & h_1 +\ldots +h_{t-1}+1 & r +1 \end{pmatrix}.
\end{equation*}
For the partition of $\Pi = \begin{bmatrix}\Pi_{11} & \Pi_{12} \\ \Pi_{21} & \Pi_{22}\end{bmatrix}$ with $\Pi_{11}\in \mathbb{R}^{t\times r}$ and $\Pi_{12}\in \mathbb{R}^{t\times (n-r)}$, we know that $\Pi_{12} = 0$.
Usually, the randomized SVD of $A$ is given in the form of full rank decomposition $A\approx U_1 \Sigma_1 V_1^{\mathrm T}$. It is fortunate that we can use $U_1$, $\Sigma_1$ and $V_1$ to obtain the approximate core problem $\{b_1,~A_{11}\}$ and do not need to generate the full orthogonal matrices $U$ and $V$ explicitly.
The approximate solution can be retrieved from $V_1$, $S_1,\ldots,S_t$ and $\Pi_{11}$.
In detail, we have
\begin{equation*}
  \Pi_{11}\begin{bmatrix}
            S_1 &  &  \\
             & \ddots &  \\
             &  & S_t
          \end{bmatrix}U_1^{\mathrm T} A V_1
          \begin{bmatrix}
            S_1 &  &  \\
             & \ddots &  \\
             &  & S_t
          \end{bmatrix}\Pi_{11}^{\mathrm T} = \Sigma_1, \quad
          \Pi_{11}\begin{bmatrix}
            S_1 &  &  \\
             & \ddots &  \\
             &  & S_t
          \end{bmatrix}U_1^{\mathrm T} b =
          \begin{bmatrix}
            \varphi_1 \\
            \vdots \\
            \varphi_t
          \end{bmatrix},
\end{equation*}
and $\varphi_{t+1} = \left\|b-U_1U_1^{\mathrm T} b\right\|$.
Since $\Sigma_1$ is computed by a randomized algorithm, there is little chance to have multiple singular values.
It is reasonable to assume the generic case with $h_j = 1$ ($j = 1,\ldots,t$), $r= t$ and
\begin{equation}\label{A11b1}
  A_{11} = \begin{bmatrix}
  \Sigma_1 \\
  0
\end{bmatrix} = \begin{bmatrix}\sigma_1 && \\ & \ddots &\\ && \sigma_r \\ && 0 \end{bmatrix},
\quad b_1 = \begin{bmatrix}
                            U_1^{\mathrm T} b \\
                            \left\|b-U_1 U_1^{\mathrm T} b\right\|
                          \end{bmatrix}=\begin{bmatrix}\varphi_1\\ \vdots \\ \varphi_r \\ \varphi_{r+1}\end{bmatrix}.
\end{equation}

The computation for small-scale TLS problems can be simplified without the SVD of $[A_{11},~b_1]$ if $A_{11}$ is diagonal as in the form of \eqref{A11b1}. We summarize that in the following lemma.

\begin{lem}\label{lem_svdA}
Consider the close form of the small scale TLS problem \eqref{close_1}.
  If $A_{11}$ and $b_1$ are in the form of \eqref{A11b1},  then the analytic solution of \eqref{close_1} is given by
  \begin{equation}\label{y_ana}
    y= \begin{bmatrix}
      \frac{\sigma_1\varphi_1}{\sigma_1^2-\|C^{-1}\|^{-2}} &
      \cdots &
      \frac{\sigma_r\varphi_r}{\sigma_r^2-\|C^{-1}\|^{-2}}
    \end{bmatrix}^{\mathrm T},
  \end{equation}
  where \[
  C =\begin{bmatrix}
          A_{11} &  b_1
            \end{bmatrix}= \begin{bmatrix}\sigma_1 && &\varphi_1\\ & \ddots &&\vdots \\ && \sigma_r &\varphi_r \\ && 0 & \varphi_{r+1}\end{bmatrix}, \quad
  C^{-1} =\begin{bmatrix}\frac1{\sigma_1} &&&-\frac{\varphi_1}{\sigma_1\varphi_{r+1}}\\ & \ddots &&\\ &&\frac 1 {\sigma_r} &-\frac {\varphi_r} {\sigma_r\varphi_{r+1}} \\ &&& \frac 1{\varphi_{r+1}}\end{bmatrix}.\]
\end{lem}

\begin{proof}
For the small scale TLS problem \eqref{close_1}, denote the augmented matrix $[A_{11},~b_1]$ as $C$, we have
\begin{equation*}
C\begin{bmatrix} 1 && &-\frac{\varphi_1}{\sigma_1} \\ & \ddots && \vdots \\&&1&-\frac{\varphi_r}{\sigma_r} \\ &&&1 \end{bmatrix}=
\begin{bmatrix}\sigma_1 &&&\\ & \ddots &&\\ &&\sigma_r & \\ &&& \varphi_{r+1}\end{bmatrix}, \quad
C= \begin{bmatrix}\sigma_1 && &\varphi_1\\ & \ddots &&\vdots \\ && \sigma_r &\varphi_r \\ && 0 & \varphi_{r+1}\end{bmatrix}.
        \end{equation*}
We have the exact expression
 \[C^{-1} =
\begin{bmatrix} 1 && &-\frac{\varphi_1}{\sigma_1} \\ & \ddots && \vdots \\&&1&-\frac{\varphi_r}{\sigma_r} \\ &&&1 \end{bmatrix}
\begin{bmatrix}\frac1{\sigma_1} &&&\\ & \ddots &&\\ &&\frac 1 {\sigma_r} & \\ &&& \frac 1{\varphi_{r+1}}\end{bmatrix}
=\begin{bmatrix}\frac1{\sigma_1} &&&-\frac{\varphi_1}{\sigma_1\varphi_{r+1}}\\ & \ddots &&\\ &&\frac 1 {\sigma_r} &-\frac {\varphi_r} {\sigma_r\varphi_{r+1}} \\ &&& \frac 1{\varphi_{r+1}}\end{bmatrix}.\]
The smallest singular value of $C$ can be obtained by $\|C^{-1}\|^{-1}$, i.e., $\sigma_{r+1}(C) = \|C^{-1}\|^{-1}$. The analytical solution of the small-scale TLS problem \eqref{close_1} is given by
\begin{equation*}
  y = \left(\Sigma_1^{\mathrm T} \Sigma_1 - \sigma_{r+1}^2\left(C\right) I_r\right)^{-1}\Sigma_1^{\mathrm T} (U_1^{\mathrm T} b)
  = \begin{bmatrix}
      \frac{\sigma_1\varphi_1}{\sigma_1^2-\|C^{-1}\|^{-2}} &
      \cdots &
      \frac{\sigma_r\varphi_r}{\sigma_r^2-\|C^{-1}\|^{-2}}
    \end{bmatrix}^{\mathrm T}.
\end{equation*}
\end{proof}

Then the approximate solution $\hat{x}$ of \eqref{TLS} is
\begin{eqnarray}
  \hat{x}&=& V_1 y = V_1\left(\Sigma_1^{\mathrm T} \Sigma_1 - \sigma_{r+1}^2\left(\begin{bmatrix}
                                                    \Sigma_1 & U_1^{\mathrm T} b \\
                                                    0 & \left\|b-U_1 U_1^{\mathrm T} b\right\|
                                                  \end{bmatrix}\right) I_r\right)^{-1}\Sigma_1^{\mathrm T} (U_1^{\mathrm T} b)\nonumber \\
     & =&V_1 \begin{bmatrix}
      \dfrac{\sigma_1\varphi_1}{\sigma_1^2-\|C^{-1}\|^{-2}} &
      \cdots &
      \dfrac{\sigma_r\varphi_r}{\sigma_r^2-\|C^{-1}\|^{-2}}
    \end{bmatrix}^{\mathrm T}. \label{x_appx}
\end{eqnarray}
If the exact SVD of $A$ is known, then the analytical solution of TLS problem within $\{b, ~A\}$ is given in \eqref{x_appx} with $r=n$ and $V_1 = V$. We present the respective algorithm as Algorithm~\ref{algo_tls2}.

\begin{algorithm}[ht]
\caption{Randomized TLS with randomized $A$}\label{algo_tls2}
\begin{algorithmic}[1]
\Require
 $A\in \mathbb{R}^{m\times n}$,  $b\in \mathbb{R}^{n}$.
\Ensure
Orthonormal $U_1 \in \mathbb{R}^{m\times r}$ and  $V_1\in \mathbb{R}^{n\times r}$, diagonal matrix $\Sigma_1\in \mathbb{R}^{r\times r}$, approximate solution  $\hat{x}$.
\State Compute randomized SVD, $A\approx U_1\Sigma_1 V_1^{\mathrm T}$:
\State \quad The $m\times r$ orthonormal matrix $Q_1$ selected by Algorithm \ref{algo_tls1};
\State \quad Form the $r\times n$ matrix $Q_1^{\mathrm T} A$; Compute the SVD of $Q_1^{\mathrm T} A= U_1\Sigma_1V_1^{\mathrm T}$; $U_1$ is updated by $U_1=Q_1U_1$.
\State Solve the small-scale TLS problem \[\begin{bmatrix}
                                            \Sigma_1 \\
                                            \bf{0}
                                          \end{bmatrix} y \approx \begin{bmatrix}
                                                                    U^{\mathrm T} b \\
                                                                    \left\|b-U_1 U_1^{\mathrm T} b\right\|
                                                                  \end{bmatrix},\]
                                          with $y$ in the form of \eqref{y_ana}.
\State Form the approximate solution $\hat{x}=  V_1y$.
\end{algorithmic}
\end{algorithm}

\begin{rem}
  In this case, if the coefficient matrix $A$ is numerically low-rank with the approximate SVD (e.g. randomized SVD) $A \approx U_1 \Sigma_1 V_1^{\mathrm T}\triangleq A_r$, the approximate TLS solution of problem \eqref{TLS} is in the form of \eqref{x_appx}.
It is easy to check that
\begin{equation*}\label{eq_Ab_s_min}
\sigma_{r+1}\left(\begin{bmatrix}
                              A_r & b
                            \end{bmatrix}\right)
 = \sigma_{r+1}\left(\begin{bmatrix}
       \Sigma_1  & U_1^{\mathrm T} b \\
        \mathbf{0} &  U_2^{\mathrm T} b
     \end{bmatrix}\right)
                       =\sigma_{r+1}\left(\begin{bmatrix}
                                   \Sigma_1 & U_1^{\mathrm T} b \\
                                   0 & \left\|b-U_1U_1^{\mathrm T}  b\right\|
                                 \end{bmatrix}\right)=\sigma_{r+1}(A_{11},~b_1),
\end{equation*}
and the approximate solution $\hat{x}$ in \eqref{x_appx} can be treated as the minimum norm solution of
\begin{equation*}\label{TLS_r}
  \left[A_r^{\mathrm T} A_r - \sigma_{r+1}^2([A_r, ~b]) I_n \right] x = A_r^{\mathrm T} b.
\end{equation*}
\end{rem}

\begin{rem}
  Since the construction of a core problem within  $AX\approx B$ with multiple right-hand sides is based on the SVD of $A$ \cite{hnetynkova2013core},
  the approximate core problem may be generalized for multiple right-hand sides using the randomized SVD of $A$.
  This process is more complicated and we leave it for the future.
\end{rem}

\section{Error Analysis}\label{sec_error}

Based on the perturbation analysis of linear system in \cite[Section~2.6.2]{Golub2013matrix} and \cite[Section 1.4]{deif1986sensitivity}, we derive the error analyses for the randomized TLS algorithms.
For the sensitivity analysis of TLS problem \eqref{TLS_close}, the smallest singular value $\sigma_{n}^2(A) - \sigma_{n+1}^2([A,~b])$, or $\|(A^{\mathrm T} A - \sigma_{n+1}^2([A,~b]) I_n)^{-1}\|$ is crucial \cite[Corollary~1]{Baboulin2011}. After the randomized projection in Algorithm \ref{algo_tls2}, the smallest singular values of $A$ are discarded, thus the condition number is improved simultaneously, i.e., $\|(A_{11}^{\mathrm T} A_{11} - \sigma_{r+1}^2([A_{11},~b_1])I_r)^{-1}\|\leq \|(A^{\mathrm T} A - \sigma_{n+1}^2([A,~b]) I_n)^{-1}\|$. When required we may use a restart strategy to remove the ill-conditioning by perturbing the parameter $r$ slightly.
\begin{lem}\emph{\cite[Section~2.6.2]{Golub2013matrix}}\label{lem_perturb}
Let $A\in \mathbb{R}^{n\times n}$ be nonsingular and consider the equation $Ax=b$. If $A$ and $b$ are perturbed by infinitesimal $\Delta A$ and $\Delta b$, the solution $x$ changes by infinitesimal $\Delta x$, where
\begin{equation*}
(A+\Delta A) (x+\Delta x) = b+\Delta b.
\end{equation*}
If the spectral radius of $A^{-1}\Delta A$ is less than unity, we obtain that, upon neglecting second order terms,
\begin{equation*}
\Delta x \approx A^{-1} [\Delta b- \Delta A x], 
\quad
  \frac{\|\Delta x\|}{\|x\|} \leq\left\|A^{-1}\right\| \left(\frac{\|\Delta b\|}{\|x\|}+\| \Delta A\|\right).
\end{equation*}
\end{lem}

The theorem is given below.
\begin{thm}
  Suppose that $x_*$ and $\hat{x}$ are the solution of the TLS problem within $\{b,~A\}$ and $\{b,~A_r\}$ respectively, i.e., the minimum norm solutions of
  \begin{eqnarray*}
    (A^{\mathrm T} A -\sigma^2_{n+1}([A,~b])I_n)x = A^{\mathrm T} b, \quad
    (A_r^{\mathrm T} A_r - \sigma^2_{r+1}([b, ~A_r])I_n)x= A_r^{\mathrm T} b,
  \end{eqnarray*}
  where $A_r$ is obtained by the randomized SVD in Algorithm~\ref{algo_tls2}, satisfying $\|A-A_r\|\leq \sqrt{1+k\epsilon^2}\sigma_{k+1}(A)$ by Lemma~\ref{lem_Qgu}.
  Then error of the approximate solution can be bounded as follows:
  \begin{equation*}
    \frac{\|\hat{x}-x_*\|}{\|x_*\|}\leq \left\|(A^{\mathrm T} A -\sigma^2_{n+1}([A,~b])I_n)^{-1}\right\|\left[ \left(\frac{\|b\|}{\|x_*\|}
    +2\|A\|\right)\sqrt{1+k\epsilon^2} + 2\sigma_{k+s}(A)\right]\sigma_{k+1}(A),
  \end{equation*}
with probability not less than $1-\Delta$. Here
$\epsilon=\mathcal{C}_\Delta(\sigma_{k+1+s-p}(A) /\sigma_k(A))^{2q}$
 with $\mathcal{C}_\Delta$  defined in \eqref{c_delta}.
\end{thm}

\begin{proof}
For the perturbation of the coefficient matrix, denote
\begin{equation*}
\begin{split}
 \Delta A &=(A^{\mathrm T} A -\sigma^2_{n+1}([A,~b])I_n)-(A_r^{\mathrm T} A_r - \sigma^2_{r+1}([A_r,~b])I_n)\\
 &= (A^{\mathrm T} A-A_r^{\mathrm T} A_r)+\left(\sigma^2_{r+1}([A_r,~b])-\sigma^2_{n+1}([A,~b])\right)I_n,
 \end{split}
\end{equation*}
 The matrix $A_r$ is obtained by the randomized SVD in Algorithm \ref{algo_tls2}, so $A_r= QQ^{\mathrm T} A$. We have $\|A_r\|\leq \|A\|$ and $\sigma_r(A_r) \leq \sigma_r(A)$ by the interlacing property \cite[Theorem 1]{thompson1972principal}.
Then we can compute that
  \begin{eqnarray*}
    \left\|A^{\mathrm T} A - A_r^{\mathrm T} A_r\right\|  &=&\left\|A^{\mathrm T} (A-A_r) + (A-A_r)^{\mathrm T} A_r\right\|\leq \|A-A_r\|\cdot(\|A\|+\|A_r\|)
    \leq 2\|A\|\cdot\|A-A_r\|.
  \end{eqnarray*}
 Again by the interlacing property \cite[Theorem 1]{thompson1972principal}, we have
  \begin{eqnarray*}
\sigma^2_{r+1}([A_r,~b])-\sigma^2_{n+1}([A,~b])& \!=\!& (\sigma_{r+1}([A_r,~b])-\sigma_{n+1}([A,~b]))(\sigma_{r+1}([A_r,~b])+\sigma_{n+1}([A,~b])) \\
&\!\leq\!&   \sigma_{r}(A_r) (\sigma_{r}(A_r)+\sigma_{n}(A)) \leq 2\sigma_r^2(A).
  \end{eqnarray*}
  Then we get
    $\|\Delta A\|\leq  2\|A\|\cdot\|A-A_r\|  + 2\sigma^2_{r}(A)$.

 The perturbation of the right-hand side $\Delta b = A^{\mathrm T} b-A_r^{\mathrm T} b$ satisfies
  \begin{equation*}
  \|\Delta b\|=\|A^{\mathrm T} b-A_r^{\mathrm T} b\|\leq \|A-A_r\|\cdot \|b\|.
  \end{equation*}
  Since $A_r$ is obtained by the randomized SVD of $A$ and satisfies $\|A-A_r\|\!\leq\!\! \sqrt{1+k\epsilon^2}\sigma_{k+1}(A)$ by \eqref{est_Qnew},
   from Lemma \ref{lem_perturb},
  we can get
  \begin{eqnarray*}
    \frac{\|\hat{x}-x_*\|}{\|x_*\|} &\leq&
    \left\|(A^{\mathrm T} A -\sigma^2_{n+1}([A,~b])I_n)^{-1}\right\|
    \left[\left(\frac{\|b\|}{\|x_*\|} + 2\|A\|\right)\|A-A_r\| + 2\sigma^2_{r}(A)\right]\\
     &\leq& \left\|(A^{\mathrm T} A -\sigma^2_{n+1}([A,~b])I_n)^{-1}\right\|
    \left[\left(\frac{\|b\|}{\|x_*\|} + 2\|A\|\right)\sqrt{1+k\epsilon^2}\sigma_{k+1}(A) + 2\sigma^2_{r}(A)\right]\\
    &\leq& \left\|(A^{\mathrm T} A -\sigma^2_{n+1}([A,~b])I_n)^{-1}\right\|
    \left[\left(\frac{\|b\|}{\|x_*\|} + 2\|A\|\right)\sqrt{1+k\epsilon^2} + 2\sigma_{r}(A)\right]\sigma_{k+1}(A),
  \end{eqnarray*}
  with probability not less than $1-\Delta$.
  Here we have
  \begin{small}
  \begin{equation*}
  \begin{split}
    \epsilon&=\mathcal{C}_\Delta\left(\frac{\sigma_{k+1+s-p}(A)}{\sigma_k(A)}\right)^{2q}
       =
  \frac{e\sqrt{k+s}}{p+1}\left(\frac{2}{\Delta}\right)^{\frac{1}{p+1}}
  \left(\sqrt{n-k-s+p}+\sqrt{k+s}+\sqrt{2\log\frac{2}{\Delta}}\right)\left(\frac{\sigma_{k+1+s-p}(A)}{\sigma_k(A)}\right)^{2q},
  \end{split}
  \end{equation*}
  \end{small}
 with $\mathcal{C}_\Delta$ defined in \eqref{c_delta}.
\end{proof}

\begin{rem}
  From Algorithm \ref{algo_tls2}, $A_r$ is obtained by the randomized SVD of $A$ with $A_r =U_1\Sigma_1 V_1^{\mathrm T}=Q_1Q_1^{\mathrm T}A$. So the TLS solution $\hat{x}$ within $\{b,~A_r\}$ is exactly the approximate solution in Algorithm~\ref{algo_tls2}.
  The parameter $\epsilon$ will be small enough if the singular values of $A$ decay fast or we may use larger parameter $q$ to accelerate the decay.
  The theorem shows that if the coefficient matrix $A$ is of numerical low-rank, i.e., there exist $k$ such that $\sigma_{k+1}(A)$  is small enough, a good approximate solution of the TLS problem can be obtained by randomized algorithms.
  Various approximation properties relay on the fast decay of the singular values.
\end{rem}

Note that the randomized core reduction is a regularization method, so we prove that this method can give a good estimation of the truncated TLS solution $\tilde{x}$ in \eqref{x_trun}, if $\sigma_{k+1}(A)$ is small enough.
We quote the classical perturbation result on the Moore-Penrose inverse in the following lemma.

\begin{lem}\emph{\cite[Theorem 2.1]{wedin1973perturbation}}\label{lem_pseudoinverse}
 Take $T=B-A$, then
 \begin{equation}\label{pseudoBA}
   B^\dag -A^\dag = -B^\dag T A^\dag + (B^{\mathrm T} B)^\dag T^{\mathrm T} (I-AA^\dag) + (I-B^\dag B)T^{\mathrm T}(AA^{\mathrm T})^\dag.
 \end{equation}
\end{lem}
This lemma doesn't require the equal rank of matrices $A$ and $B$. In the truncated TLS, we choose $k+1$ as the truncation parameter of matrix $[A,~b]$ since $\sigma_{k+1}(A)$ is assumed to be small.

\begin{thm}\label{thm_core_ttls}
  Suppose that $\hat{x}$ and $\tilde{x}$ are the solutions of the randomized core reduction and truncated TLS problem  within original $\{b,~A\}$ respectively, i.e., the minimum norm solutions of
  \begin{eqnarray*}
    (A_r^{\mathrm T} A_r - \sigma_{r+1}^2([A_r, ~b]) I_n )x =  A_r^{\mathrm T} b, \quad
    \tilde{A}^{\mathrm T} \tilde{A} x = \tilde{A}^{\mathrm T} b,
  \end{eqnarray*}
  where $A_r$ is obtained by randomized SVD in Algorithm \ref{algo_tls2}, satisfying $\|A-A_r\|\leq \sqrt{1+k\epsilon^2}\sigma_{k+1}(A)$ by Lemma \ref{lem_Qgu} and $[\tilde{A},~\tilde{b}] = \sum_{i=1}^{k+1}\bar{u}_i\bar{\sigma}_i\bar{v}_i^{\mathrm T}$ from the SVD of $[A,~b]$.
  Then error of the approximate solution can be bounded as follows,
\begin{equation*}
 \frac{\|\hat{x}-\tilde{x}\|}{\|\tilde{x}\|}  \leq \frac{\sigma_{k+1}(A)}{\sigma_{r}^2(A_r)-\sigma_{r+1}^2([A_r, b])}
     \left[c_1\frac{\|b\|}{\|\tilde{x}\|} + c_2\|A\|\left(1+\frac{\|A\|\|\tilde{r}\|/\|\tilde{x}\|}{\sigma_{r}^2(A_r)-\sigma_{r+1}^2([A_r, b])}\right)\right],
\end{equation*}
with probability not less than $1-\Delta$. Here
$\epsilon=\mathcal{C}_\Delta(\sigma_{k+1+s-p}(A) /\sigma_k(A))^{2q}$
 and $\mathcal{C}_\Delta$ defined in \eqref{c_delta}.
 Furthermore, the residual $\hat{r}= b-A\hat{x}$ of the randomized core reduction is estimated as follows,
 \begin{equation*}
   \|\hat{r}\|\leq   c_1\sigma_{k+1}(A)\sqrt{1+\|\hat{x}\|^2},
 \end{equation*}
 with the constants $c_1 = 1+\sqrt{1+k\epsilon^2}$ and  $c_2 = 2c_1+1$.
\end{thm}
\begin{proof}
According to the randomized core reduction (Algorithm \ref{algo_tls2}), the approximate solution is given by
\begin{equation*}
  \hat{x} =  \left[A_r^{\mathrm T} A_r - \sigma_{r+1}^2([A_r, ~b]) I_n \right]^{-1} A_r^{\mathrm T} b,
\end{equation*}
and the truncated TLS solution $\tilde{x} = \tilde{A}^\dag b$.
Then by Lemma \ref{lem_pseudoinverse},
\begin{equation*}
\begin{split}
  \hat{x}-\tilde{x} &= \left[A_r^{\mathrm T} A_r - \sigma_{r+1}^2([A_r, ~b]) I_n  \right]^{-1} A_r^{\mathrm T} b - (\tilde{A}^{\rm T}\tilde{A})^\dag \tilde{A}^{\rm T} b \\
  &= \left[A_r^{\mathrm T} A_r - \sigma_{r+1}^2([A_r, ~b]) I_n  \right]^{-1} (A_r-\tilde{A})^{\mathrm T} b +
  \left[(A_r^{\mathrm T} A_r - \sigma_{r+1}^2([A_r, ~b]) I_n  )^{-1}- (\tilde{A}^{\mathrm T} \tilde{A})^\dag\right] \tilde{A}^{\mathrm T} b \\
  &\overset{\eqref{pseudoBA}}
  {=} \left[A_r^{\mathrm T} A_r - \sigma_{r+1}^2([A_r, ~b]) I_n  \right]^{-1} (A_r-\tilde{A})^{\mathrm T} b \\
  & \quad +\left[A_r^{\mathrm T} A_r - \sigma_{r+1}^2([A_r, ~b]) I_n  \right]^{-1}T\tilde{x}-\left[A_r^{\mathrm T} A_r - \sigma_{r+1}^2([A_r, ~b]) I_n  \right]^{-2}T\tilde{A}^{\rm T}\tilde{r},
  \end{split}
\end{equation*}
with $T = \tilde{A}^{\mathrm T} \tilde{A}-A_r^{\mathrm T} A_r + \sigma_{r+1}^2([A_r, ~b]) I_n$ and the residual for truncated TLS problem $\tilde{r} = b- \tilde{A}\tilde{x}$.

By $\|\tilde{A}\|\leq \|A\|$, $\|A_r\|\leq \|A\|$, $\sigma_{r+1}([A_r,~b])\leq \sigma_{k+1}(A)$, and
\begin{equation*}
\begin{split}
  \|A-\tilde{A}\|&\leq \|[A, b]-[\tilde{A}, \tilde{b}]\| =\sigma_{k+2}([A,~b])\leq \sigma_{k+1}(A),\\
\|A_r-\tilde{A}\|&\leq \|A-A_r\| + \|A-\tilde{A}\|\leq c_1\sigma_{k+1}(A),\\
  \|T\|&\leq \|A_r-\tilde{A}\|(\|\tilde{A}\|+\|A_r\|) + \sigma_{r+1}^2([A_,~b])\\
 &\leq [2c_1\|A\| + \sigma_{k+1}(A)]\sigma_{k+1}(A)\leq c_2\|A\|\sigma_{k+1}(A).
\end{split}
\end{equation*}
we have
\begin{equation*}
  \begin{split}
     \frac{\|\hat{x}-\tilde{x}\|}{\|\tilde{x}\|} &\leq  \frac{\sigma_{k+1}(A)}{\sigma_{r}^2(A_r)-\sigma_{r+1}^2([A_r, b])}
     \left[c_1\frac{\|b\|}{\|\tilde{x}\|} + [2c_1\|A\|+\sigma_{k+1}(A)]\left(1+\frac{\|A\|\|\tilde{r}\|/\|\tilde{x}\|}{\sigma_{r}^2(A_r)-\sigma_{r+1}^2([A_r, b])}\right)\right]\\
     & \leq \frac{\sigma_{k+1}(A)}{\sigma_{r}^2(A_r)-\sigma_{r+1}^2([A_r, b])}
     \left[c_1\frac{\|b\|}{\|\tilde{x}\|} + c_2\|A\|\left(1+\frac{\|A\|\|\tilde{r}\|/\|\tilde{x}\|}{\sigma_{r}^2(A_r)-\sigma_{r+1}^2([A_r, b])}\right)\right].
  \end{split}
\end{equation*}

Consider the residual $\hat{r} = b-A\hat{x}$  of the randomized core reduction, we obtain
\begin{equation*}
\begin{split}
  \hat{r} &= b-A\hat{x}= b- A \left[A_r^{\mathrm T} A_r - \sigma_{r+1}^2([A_r, ~b])I_n \right]^{-1} A_r^{\mathrm T} b\\
  & = b-A_r \left[A_r^{\mathrm T} A_r - \sigma_{r+1}^2([A_r, ~b])I_n\right]^{-1} A_r^{\mathrm T} b
       + (A_r-A) \left[A_r^{\mathrm T} A_r - \sigma_{r+1}^2([A_r, ~b]) I_n\right]^{-1} A_r^{\mathrm T} b,
\end{split}
\end{equation*}
then
\begin{equation*}
\begin{split}
  \|\hat{r}\| &\leq \left\|\begin{bmatrix}
  U_1^{\mathrm T} b - \Sigma_1(\Sigma_1^{\mathrm T} \Sigma_1-\sigma_{r+1}^2([A_r, ~b])I_r)^{-1} \Sigma_1^{\mathrm T} U_1^{\mathrm T} b \\
  U_2^{\mathrm T} b
  \end{bmatrix}\right\| + \|A_r-A\|\|\hat{x}\|\\
  & =\left\|\begin{bmatrix}
  U_1^{\mathrm T} b - \Sigma_1y \\
  U_2^{\mathrm T} b
  \end{bmatrix}\right\| + \|A_r-A\|\|\hat{x}\|
  =\left\|  b_1 - A_{11}y \right\| + \|A_r-A\|\|\hat{x}\|\\
  &= \sigma_{r+1}([A_r, ~b])\sqrt{1+\|y\|^2}+ \|A_r-A\|\|\hat{x}\|\\
    &\leq  \sigma_{r+1}([A_r, ~b])\sqrt{1+\|\hat{x}\|^2}+ \sqrt{1+k\epsilon^2}\sigma_{k+1}(A)\|\hat{x}\|
    \leq c_1\sigma_{k+1}(A)\sqrt{1+\|\hat{x}\|^2}.
\end{split}
\end{equation*}
\end{proof}

\begin{rem}
  While the bound  for the error in Theorem \ref{thm_core_ttls} is pessimistic, it gives an indication of the influence of $\sigma_{k+1}(A)$. The bound for the residual seems reasonalbe, because it coincides with the minimization of $\|r\|/\sqrt{1+\|x\|^2}$ in the TLS problem and this term $\|\hat{r}\|/\sqrt{1+\|\hat{x}\|^2}$ is relatively small for the ill-posed problems.
\end{rem}

\section{Numerical Experiments}\label{sec_numerical}
In this section, we give several examples to illustrate that the randomized algorithms are as accurate as the classical methods.
The computations are carried out in \textsc{MATLAB}
R2015b 64-bit (with an Intel Core i5 6200U CPU  @2.30GHz 2.40GHz  processor and 8 GB RAM).
The comparison results are computed by the partial SVD \cite{larsen1998lanczos} with package PROPACK \cite{larsen2004propack}.

For a better understanding of the tables below, we list here the notation:
\begin{itemize}
  \item $x_*$ is the exact solution of the TLS problem \eqref{TLS} or \eqref{TLS_close};
  \item  ${\rm Err} = \|\hat{x}-x_*\|/\|x_*\|$ (in Algorithm \ref{algo_tls2}) is the relative error;
  \item ${\rm Time}$ is the execution time (in seconds) of the randomized core reduction in Algorithm \ref{algo_tls2};
  \item ${\rm Rank}$ stands for the number of samples, i.e., the rank of the small-scale TLS problem, which is selected by the adaptive randomized range finder in Algorithm~\ref{algo_tls1};
  \item ${\rm Err\_p}$ and ${\rm Time\_p}$ are respectively the relative errors and execution time computed by PROPACK \cite{larsen2004propack} to the $r$-th singular values with ``$r= $ Rank'';
  \item The tolerance $\varepsilon = 10^{-3}$ in Algorithm~\ref{algo_tls1} for Example~\ref{ex1d}-- \ref{ex2d}.
\end{itemize}

\begin{exam}\label{ex1d}
The collection of examples are from Hansen's Regularization Tools \cite{HansenTool}.
All the problems are derived from discretizations of Fredholm integral equations of the first kind with a square integrable kernel
\begin{equation*}
  \int_{a}^{b} K(s,t)f(t)dt = g(s), \quad c\leq s\leq d.
\end{equation*}
The right-hand side $g$ and the kernel $K$ are given, and $f$ is the unknown solution, which are extremely sensitive to high-frequency perturbations.
Two different discretization techniques are used --- the quadrature method and the Galerkin method with orthonormal basis functions.
We choose the examples \textsf{i\_laplace, shaw, heat, foxgood, phillips, gravity} as in Table \ref{hansentest}.
The decaying trends for the singular values corresponding are shown in Figure \ref{fig_decay}.
\begin{table}[ht]
  \centering
  \begin{tabular}{ll}
    \hline
\textsf{i\_laplace}  &  Inverse Laplace transformation \\
\textsf{shaw}   &  One-dimensional image restoration model \\
\textsf{heat}   & Inverse heat equation \\
\textsf{foxgood}& Severely ill-posed test problem \\
\textsf{phillips}& Phillips' ``famous'' test problem \\
\textsf{gravity} & One-dimensional gravity surveying problem  \\
    \hline
  \end{tabular}
  \caption{Test Problems in Hansen's Regularization Tools}\label{hansentest}
\end{table}

\begin{figure}[ht]
  \centering
  \includegraphics[width=10cm,height=6cm]{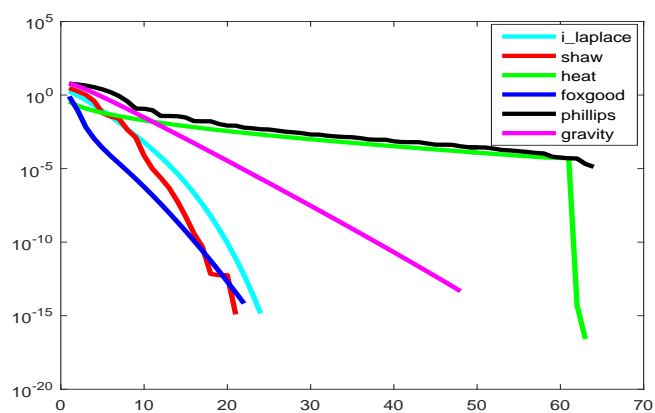}
  \caption{The decaying trends of singular values of the test examples.}\label{fig_decay}
\end{figure}
\end{exam}

The numerical results from Algorithm~\ref{algo_tls2} are shown in Tables~\ref{tab_Algo_IIe1} and \ref{tab_Algo_IIe2} and Figure \ref{fig_accu2}.
With the increasing size of the problem (compare Tables~\ref{tab_Algo_IIe1} and \ref{tab_Algo_IIe2}), the randomized TLS algorithm shows more advantages over the classical ones and the partial SVD.
For problems of the same size (Tables~\ref{tab_Algo_IIe1} or \ref{tab_Algo_IIe2}), the randomized TLS algorithm can save about 90$\%$ ``Time'' and still achieve similar errors.
It is important to remember that $y$ is the solution of the small-scale TLS problem generated by  Algorithm~\ref{algo_tls2} or PROPACK \cite{larsen2004propack}. Notice that execution
times from MATLAB may be affected by many factors, so the associated information below should be used as a rough guide only.
For Algorithm~\ref{algo_tls1} with fixed-precision, the computed rank (``Rank'' in Table~\ref{tab_Algo_IIe1}) reflects the decaying trends of singular values (in Figure~\ref{fig_decay}).
If the singular values are unknown, the randomized algorithm  works well with the adaptive range finder \cite{halko2011finding}.

\begin{table}[!ht]
\footnotesize
\centering
\begin{tabular}{c|ccccc}
$1024$&$\rm{Err} $&$\rm{Time} $& ${\rm Rank}$&\rm{Err\_p} &$\rm{Time\_p} $\\
\hline
i\_laplace &8.748E-04 & 1.12E-01 &  18 & 8.724E-04 & 1.06 \\
shaw	&1.860E-02 & 5.50E-02 &  11 & 1.860E-02 & 0.66 \\
heat	&5.688E-03 & 1.39E-01 &  66 & 5.718E-03 & 1.05 \\
foxgood	&7.717E-03 & 5.23E-02 &  10 & 7.545E-03 & 0.68 \\
phillips	&1.745E-02 & 3.12E-01 & 136 & 1.826E-02 & 2.25 \\
gravity	&6.406E-04 & 7.75E-02 &  20 & 6.388E-04 & 0.71 \\
\end{tabular}
\caption{Data with randomized SVD of $A$ for $n=1024$ and $\varepsilon=10^{-3}$.}
\label{tab_Algo_IIe1}
\end{table}

\begin{table}[!ht]
\footnotesize
\centering
\begin{tabular}{c|ccccc}
$4096$&$\rm{Err} $&$\rm{Time} $& ${\rm Rank}$&\rm{Err\_p} &$\rm{Time\_p} $\\
\hline
i\_laplace&1.439E-04 & 1.07 &  18 & 1.439E-04 & 32.74 \\
shaw	&1.853E-02 & 0.686 &  11 & 1.853E-02 & 32.38 \\
heat	&6.246E-03 & 2.66 &  62 & 6.715E-03 & 34.96 \\
foxgood&3.603E-03 & 0.670 &  11 & 3.419E-03 & 32.32 \\
phillips	&8.770E-03 & 2.88 & 133 & 7.094E-03 & 40.07 \\
gravity	&3.817E-04 & 0.961 &  19 & 3.813E-04 & 32.90 \\
\end{tabular}
\caption{Data with randomized SVD of $A$ for $n=4096$ and $\varepsilon=10^{-3}$.}
\label{tab_Algo_IIe2}
\end{table}

\begin{figure}[!ht]\centering
\subfigure[][\textsf{deriv2}]
    {\includegraphics[scale=0.3]{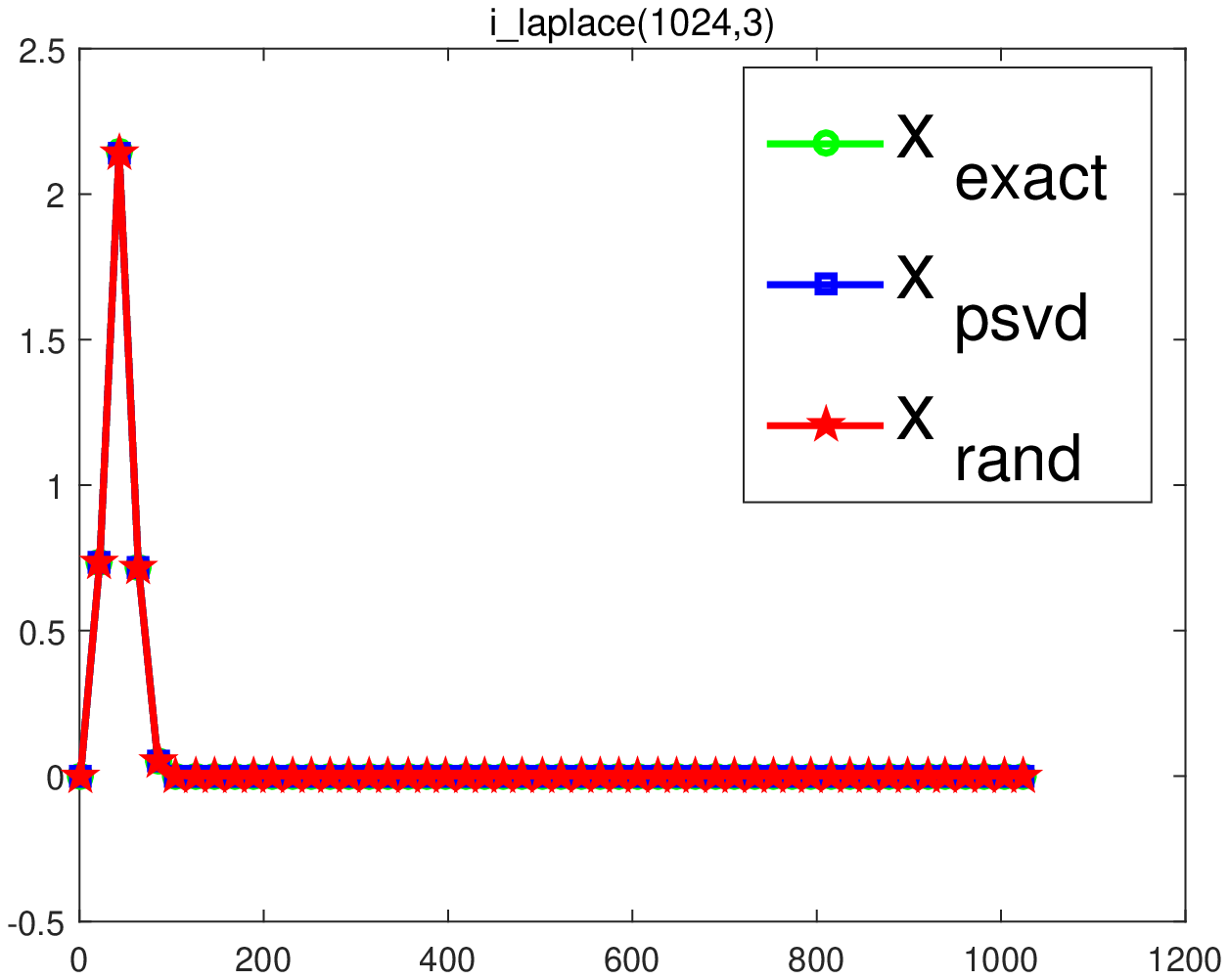}}
\subfigure[][\textsf{shaw}]
    {\includegraphics[scale=0.3]{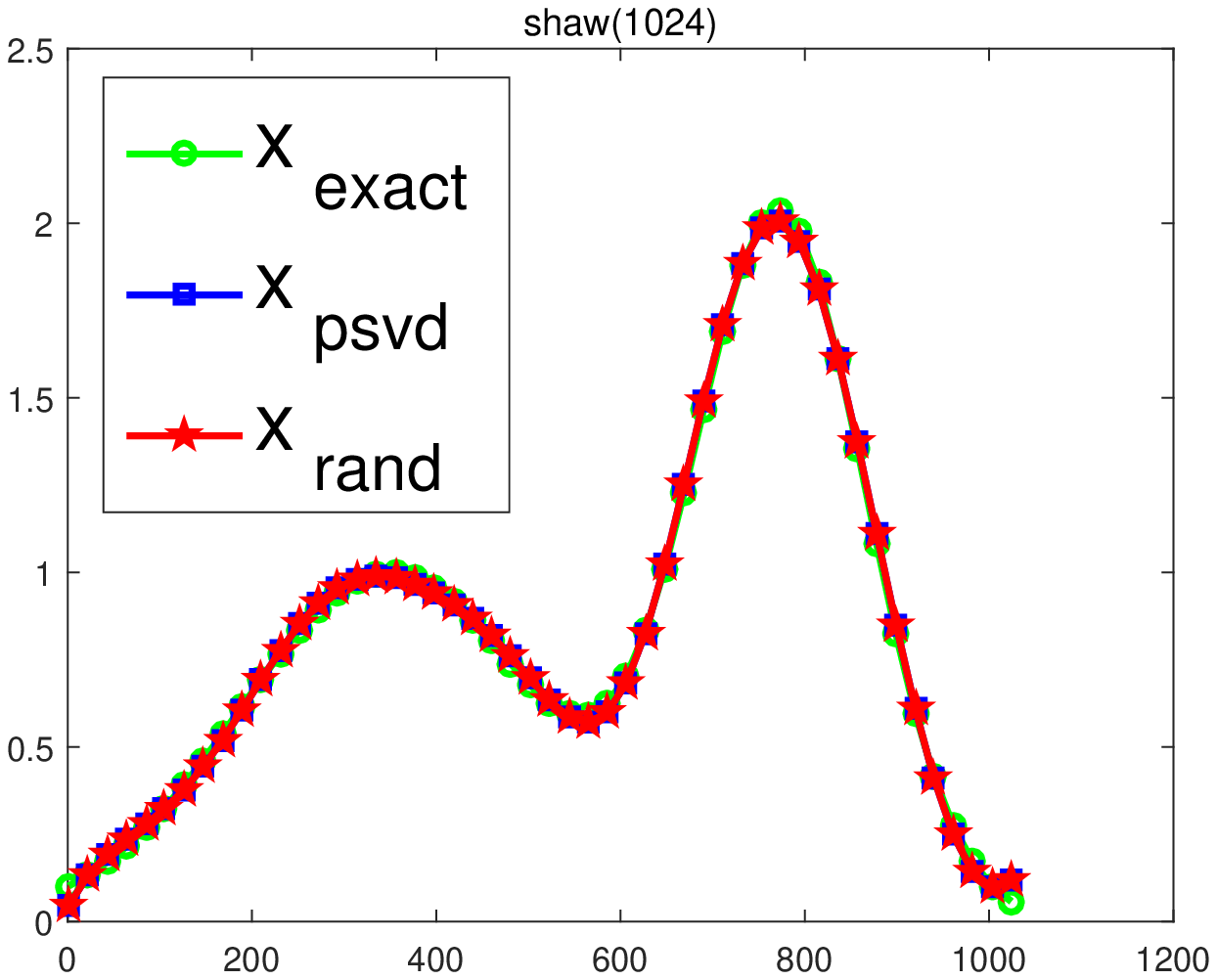}}
\subfigure[][\textsf{heat}]
    {\includegraphics[scale=0.3]{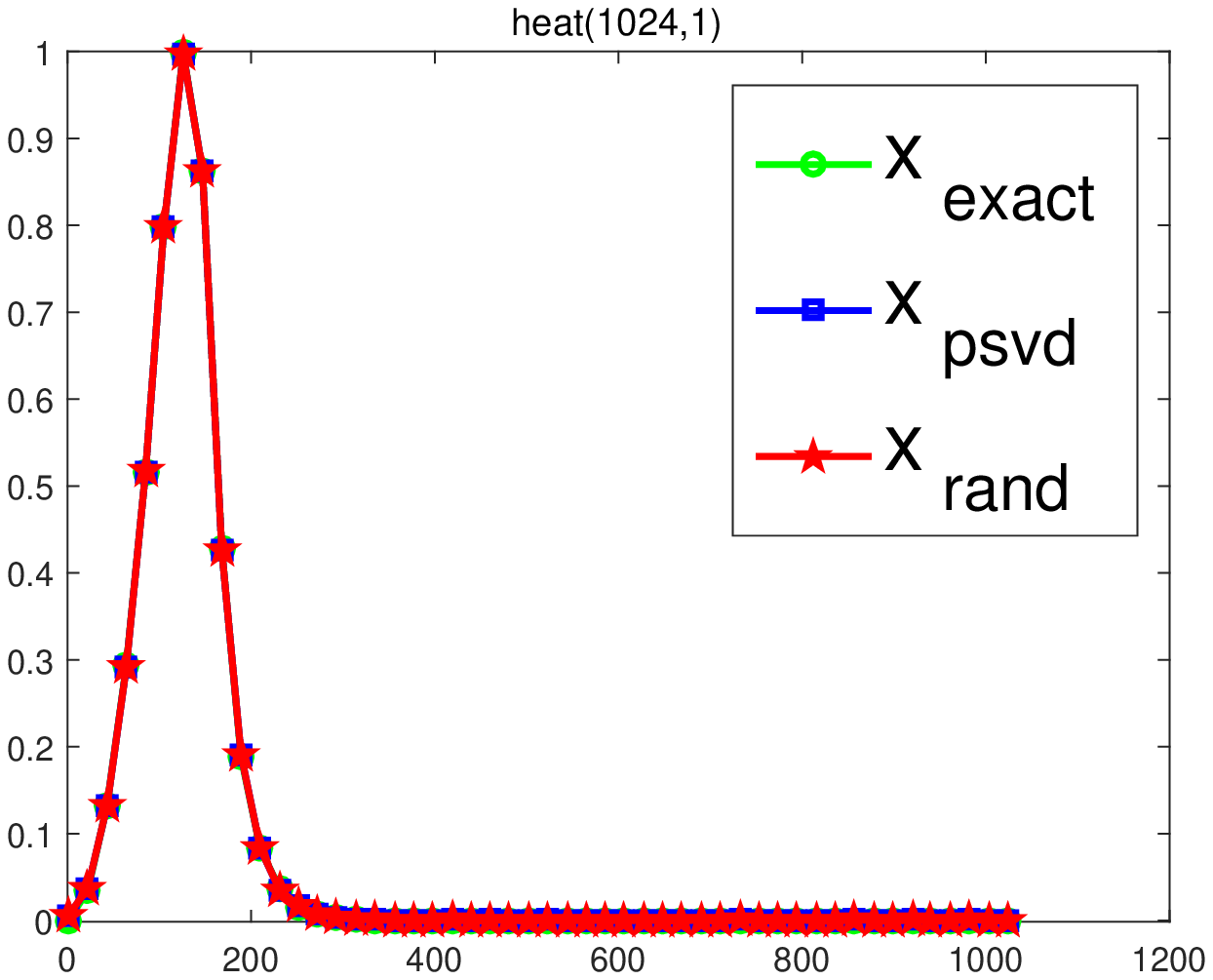}}
\subfigure[][\textsf{foxgood}]
    {\includegraphics[scale=0.3]{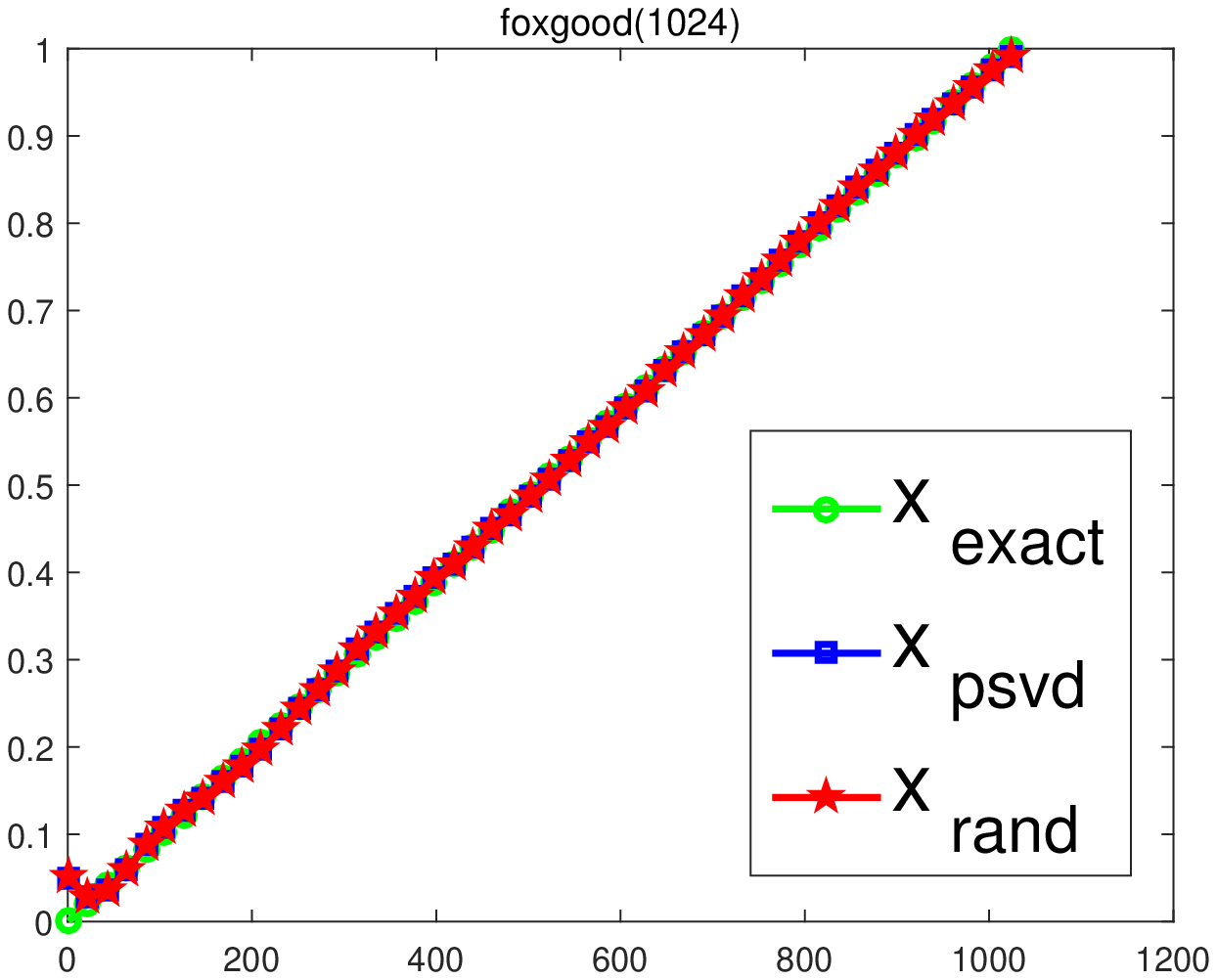}}
\subfigure[][\textsf{phillips}]
    {\includegraphics[scale=0.3]{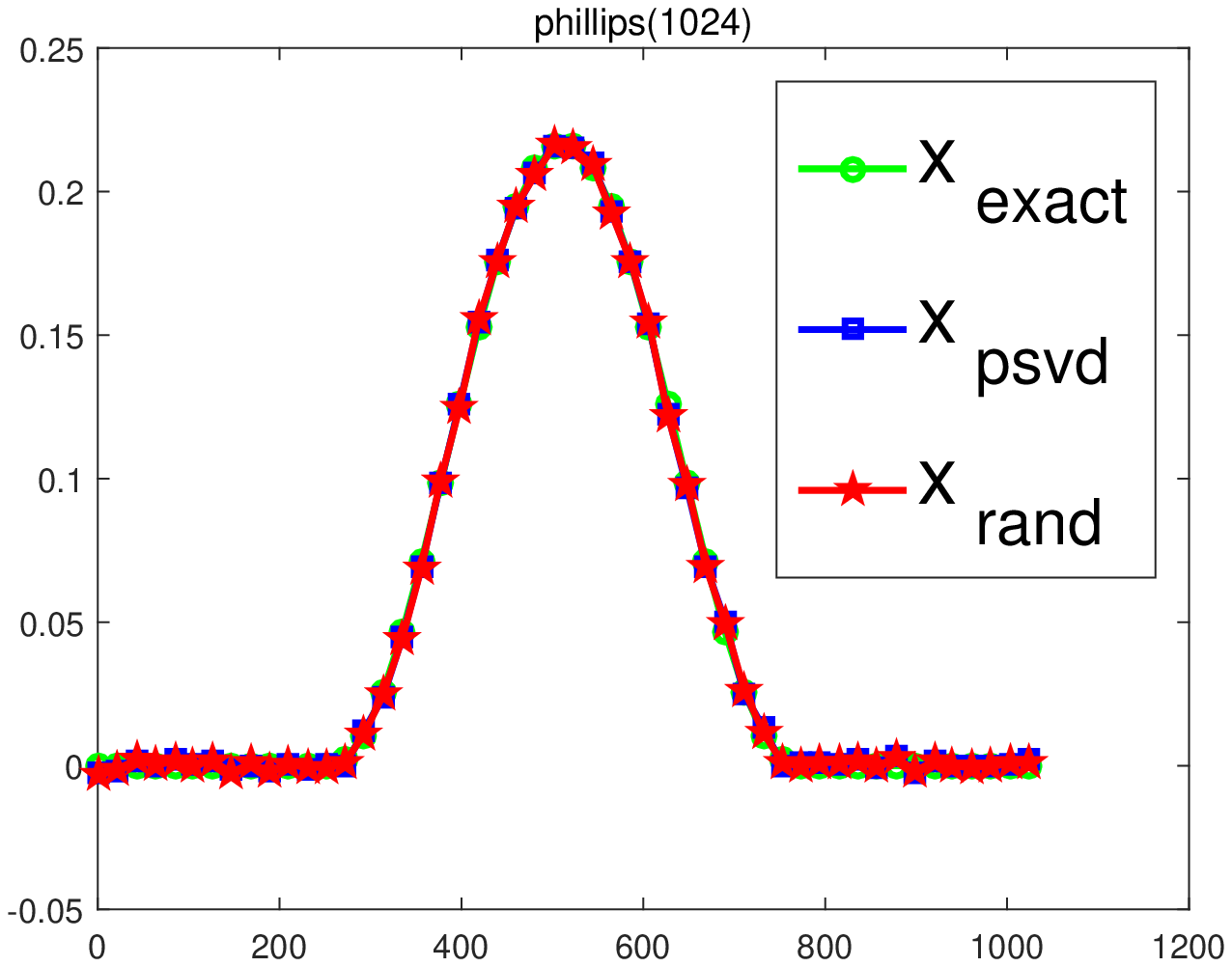}}
\subfigure[][\textsf{gravity}]
    {\includegraphics[scale=0.3]{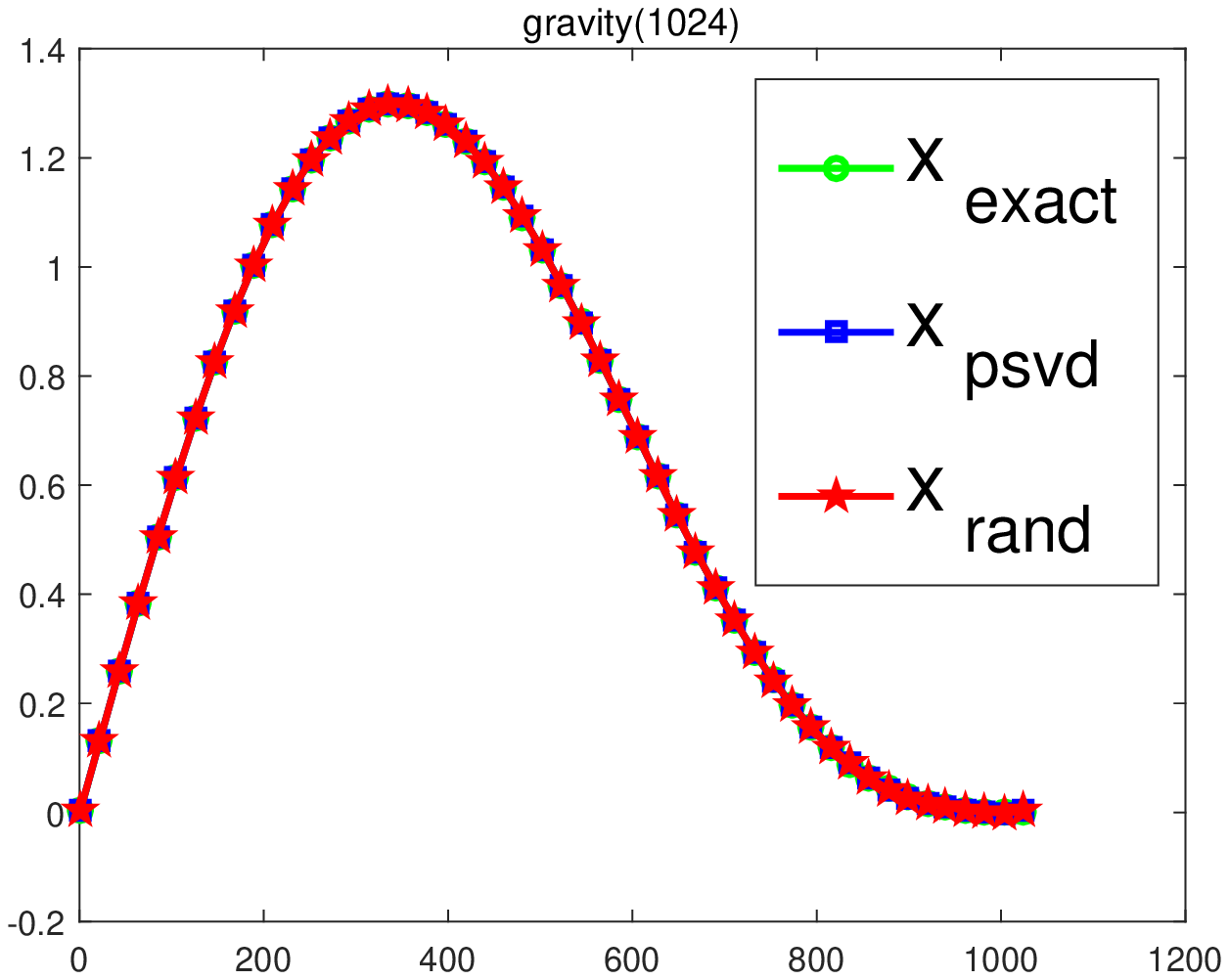}}
\caption{Solutions obtained by Algorithm \ref{algo_tls2} for $n=1024$.}\label{fig_accu2}
\end{figure}

\begin{exam}\label{ex2d}
A first kind Fredholm integral equation in two dimensions may take the form
\begin{equation*}
  \int_{a}^{b}\int_{c}^{d}K(x,y;s,t)f(s,t)dsdt = g(x,y),
\end{equation*}
where the kernel $K$ is a real convolution operator $K(x,y;s,t)=h(x-s,y-t)$.
 For example, discretization of a two dimensional model problem in gravity surveying, in which
a mass distribution $f(s,t)$ is located at depth $d$, while the vertical component of the
gravity field $g(x,y)$ is measured at the surface. The resulting problem has the kernel
\begin{equation*}
  K(x,y;s,t) = d\left[d^2+(x-s)^2 + (y-t)^2\right]^{-\frac32},
\end{equation*}
The constant $d$ controls the decay of the singular values (the larger the $d$, the faster
the decay). Let the right-hand side $g$ be given for the solution $f(s, t) = \sin(\pi t)\sin(\pi s)$.
From Algorithm \ref{algo_tls2} and the results are given in Table \ref{tab_Algo_IIe3} and Figure \ref{fig_gravity2d}.
\end{exam}

\begin{table}[!ht]
\footnotesize
\centering
\begin{tabular}{c|ccccc}
$n$&$\rm{Err} $&$\rm{Time} $& ${\rm Rank}$&\rm{Err\_p} &$\rm{Time\_p} $\\
\hline
64 & 1.310E-03 &3.95E-02 &  64 & 1.310E-03 & 0.27 \\
256 & 1.044E-01 &6.17E-02 & 209 & 1.108E-01 & 0.51 \\
1024 & 4.083E-02 &0.384 & 231 & 4.517E-02 & 8.71 \\
4096 & 1.645E-02 &3.86 & 247 & 1.620E-02 & 25.82 \\
\end{tabular}
\caption{Gravity in two dimensions with $\varepsilon=10^{-3}$.}
\label{tab_Algo_IIe3}
\end{table}
\begin{figure}
  \centering
  \includegraphics[width=12cm,height=5cm]{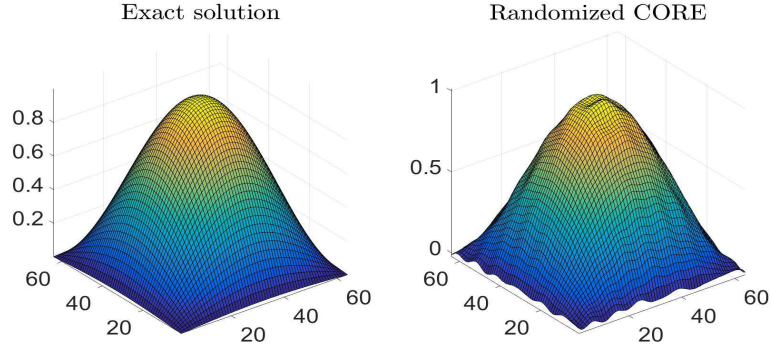}
  \caption{Two-dimensional gravity model for $n=4096$}\label{fig_gravity2d}
\end{figure}
Up to $n = 4096$, solutions can be obtained efficiently. We notice that ``Rank'' does not increase much, or information increases slowly with respect to $n$ and the randomized algorithm works well.

\begin{exam}\label{exblur}
  Consider the image blurring model
  \begin{equation*}
    A_c X A_r^{\rm T} = B,
  \end{equation*}
  where $A_r$ and $A_c$ are Toeplitz matrices \cite{hansen2006deblurring}, and blind deconvolution involves the TLS problem with  $A= A_r\otimes A_c$,  $x = {\rm vec}(X)$ and $b = {\rm vec}(B)$. We test a simple image for Algorithm \ref{algo_tls2} and the results are given in Table \ref{tab_Algo_IIe4} and Figure \ref{fig_image}.
\end{exam}

\begin{table}[!ht]
\footnotesize
\centering
\begin{tabular}{c|ccccc}
$n$&$\rm{Err} $&$\rm{Time} $& ${\rm Rank}$&\rm{Err\_p} &$\rm{Time\_p} $\\
\hline
256 & 3.880E-01 &5.63E-02 &  26 & 3.544E-01 & 0.27 \\
1024 & 2.560E-01 &0.135 & 131 & 2.566E-01 & 0.94 \\
4096 & 1.625E-01 &7.73 & 854 & 1.627E-01 & 97.94 \\
\end{tabular}
\caption{Image restoration $A$ with $\varepsilon=0.1$.}
\label{tab_Algo_IIe4}
\end{table}
\begin{figure}
  \centering
  \includegraphics[width=12cm]{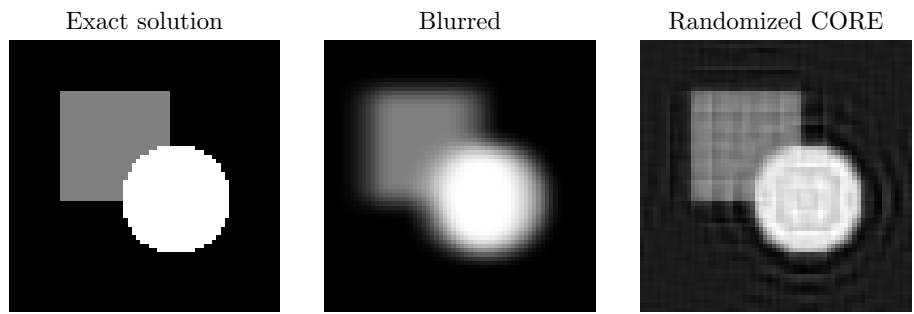}
  \caption{Image restoration for $n=4096$ and $\varepsilon=0.1$.}\label{fig_image}
\end{figure}

The matrix $A$ does not need to be constructed explicitly by Kronecker product
based on
${\rm vec}(AXB) = (B^{\rm T} \otimes A) {\rm vec}(X)$.
Furthermore for Toeplitz matrices, matrix-vector products can be accelerated by the Fast Fourier Transformation (FFT) \cite[Section 1.4.1]{Golub2013matrix}.
In this example we choose a larger tolerance $\varepsilon = 0.1$ and the results for different $n$ are shown in Table \ref{tab_Algo_IIe4}. Figure \ref{fig_image} gives the original image $X$, blurred image $B$ and the restored image by randomized core reduction, respectively. We can see that the restored image retrieves the main property of the original one. In fact, if we use $\varepsilon = 10^{-3}$, then 95.5s is needed to get a similar error ``Err = 1.210E$-$01'' with larger ``Rank = 3535'' by Algorithm~\ref{algo_tls2} for the case $n=4096$. More time and storage is occupied but the approximation is not much better than that in Table \ref{tab_Algo_IIe4}. So the large parameter $\varepsilon =0.1$ is more suitable.
  The randomized core reduction algorithm may not work well for this problem for the high ``Rank'' and
  it may be more efficient if we consider the the block Toeplitz with Toeplitz block (BTTB) structure of matrix $A$
  or utilize other randomized matrices other than Gaussian.
  The regularization of the structured TLS problem has been considered \cite{pruessner2003blind} and the corresponding structured randomized algorithm will be considered in future.

\begin{exam}\label{exIR}
  We test three severely ill-posed examples \textsf{PRdiffusion(n)}, \textsf{PRnmr(n)} and \textsf{PRblurgauss} from the IR Tools \cite{gazzolaper2018IR}.

{\rm(1)} \textsf{PRdiffusion(n)} is a 2D diffusion problem in the domain $[0, T]\times[0, 1]\times[0, 1]$:  \[\frac{\partial u}{\partial t} = \nabla^2 u\]
with homogeneous Neumann boundary conditions and a smooth function $u_0$ as initial condition at time
$t = 0$. It is generated by the statement:
\begin{equation*}
 \rm  [A, b, x, ProbInfo] = PRdiffusion(n);
\end{equation*}
where the function handle $A$ represents the PDE, the true solution $x$ and the right-hand side $b$ consist of the $N=n^2$ values of $u_0$ and $u_T$, respectively.

{\rm(2)} \textsf{PRnmr(n)} is the 2D Nuclear Magnetic Resonance (NMR) relaxometry and mathematically modeled using the following Fredholm integral equation of the first kind
        \[\int_0^{\widehat{T}^1}\int_0^{\widehat{T}^2}\kappa(\tau^1,\tau^2, T^1,T^2)f(T^1,T^2)dT^1dT^2= g(\tau^1,\tau^2),\]
where $g(\tau^1,\tau^2)$ is the noiseless signal as a function of experiment times $(\tau^1,\tau^2)$, and $f(T^1,T^2)$ is the density distribution function. The kernel  is separable:
\[\kappa(\tau^1,\tau^2, T^1,T^2) = (1-2{\rm exp}(-\tau^1/T^1)){\rm exp}(-\tau^2/T^2),\]
and, upon variable transformation, regarded as a Laplace kernel. The function is generated by:
\[\rm  [A, b, x, ProbInfo] = PRnmr(n);\]
and the function handle $A$ has a Kronecker structure.

{\rm(3)} \textsf{PRblurgauss} simulates a spatially invariant Gaussian blur, and we choose
one of the synthetically generated images that is made up of randomly placed small ``dots'', with random intensities.
This test image may be used to simulate stars being imaged from ground based telescopes. To generate the test problem, we use
\begin{equation*}
  \begin{split}
       & \rm PRoptions.trueImage = `dotk'; \\
       & \rm [A, b, x, ProbInfo] = PRblurgauss(n,PRoptions);
  \end{split}
\end{equation*}
where $x$ and $b$ are the true image and the noisy blurred image of size $n\times n$, respectively.
\end{exam}

We apply  Algorithm \ref{algo_tls2} to the three examples with size $n = 64$, then the corresponding linear system in \eqref{eq:Axb} is of the size $n^2\times n^2 = 4096\times4096$. The package PROPACK gives similar precision and costs more ``Time'', so we just list the results computed by the randomized core reduction, in Figures \ref{fig_ir1},  \ref{fig_ir2} and \ref{fig_ir3}. The tolerance $\varepsilon$  from Algorithm \ref{algo_tls1} controls the precision of the approximate solution.  From the figures we observe that with smaller tolerance $\varepsilon$,  we can approximate the solution better. Since there is some noise in the ill-posed problems, the tolerance cannot be set too small so as to avoid unstable approximate solutions. The parameter $\varepsilon$ works as the regularization parameter which is difficult to choose for different kinds of problems. For the 2D linear inverse problems, it is more difficult to obtain accurate approximate solutions than the 1D cases, and more ``Time'' is required.
The details are shown in Table \ref{tab_ir}. But it is acceptable in real applications and competitive by comparison with others' results \cite{bai2017modulus}.  In the example \textsf{PRdiffusion}, the matrix-vector products with $A$ and $A^{\rm T}$ consume more ``Time'' than the others.

For the large-scale problems from IR Tools, the matrix  $A$ is either represented sparsity, or is given in a form (i.e., a user-defined object or a function handle) for which matrix-vector products can be performed efficiently. This is consistent with our randomized core reduction, where no explicit $A$ is required.

\begin{table}[!ht]
\footnotesize
\centering
\begin{tabular}{c|ccccc}
~ & $\varepsilon$& $\rm{Err} $&$\rm{Time} $& ${\rm Rank}$& $\|b-A\hat{x}\|$ \\
\hline
\multirow{2}{*}{PRdiffusion}
  & 0.1  & 1.7955E-01   &  348.37&  61 &  7.0645E-04\\
  & 0.001 & 6.6136E-02  &  590.72&  106 &  1.3938E-06\\
\hline
\multirow{2}{*}{PRnmr}
   &0.1 &  9.2856E-01&  20.007&  86 & 5.5215E-03\\
   &1.0E-05 & 3.5569E-01& 112.13&  503  & 1.2728E-07\\
\hline
\multirow{2}{*}{PRblurgauss}
   &0.1  &  4.5548E-01&  24.747& 325 & 3.7541E-04\\
   &0.001  & 2.7952E-01& 42.484  & 535  & 1.6602E-06\\
\end{tabular}
\caption{Examples from IR Tools for $n=64$.}
\label{tab_ir}
\end{table}

\begin{figure}[!ht]\centering
\subfigure[][\textsf{Exact}]
    {\includegraphics[scale=0.15]{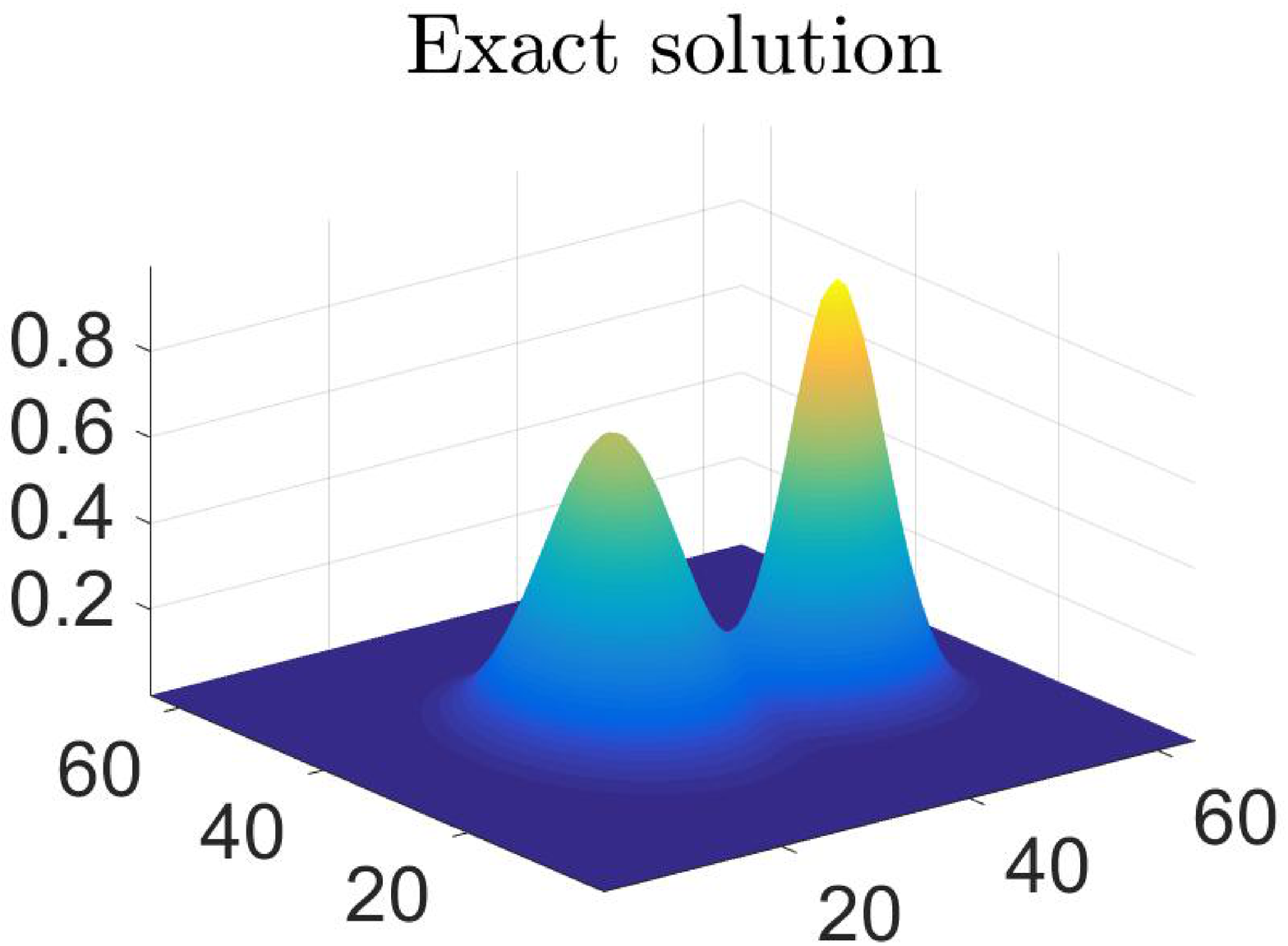}}
\subfigure[][$\varepsilon = 0.1$]
    {\includegraphics[scale=0.15]{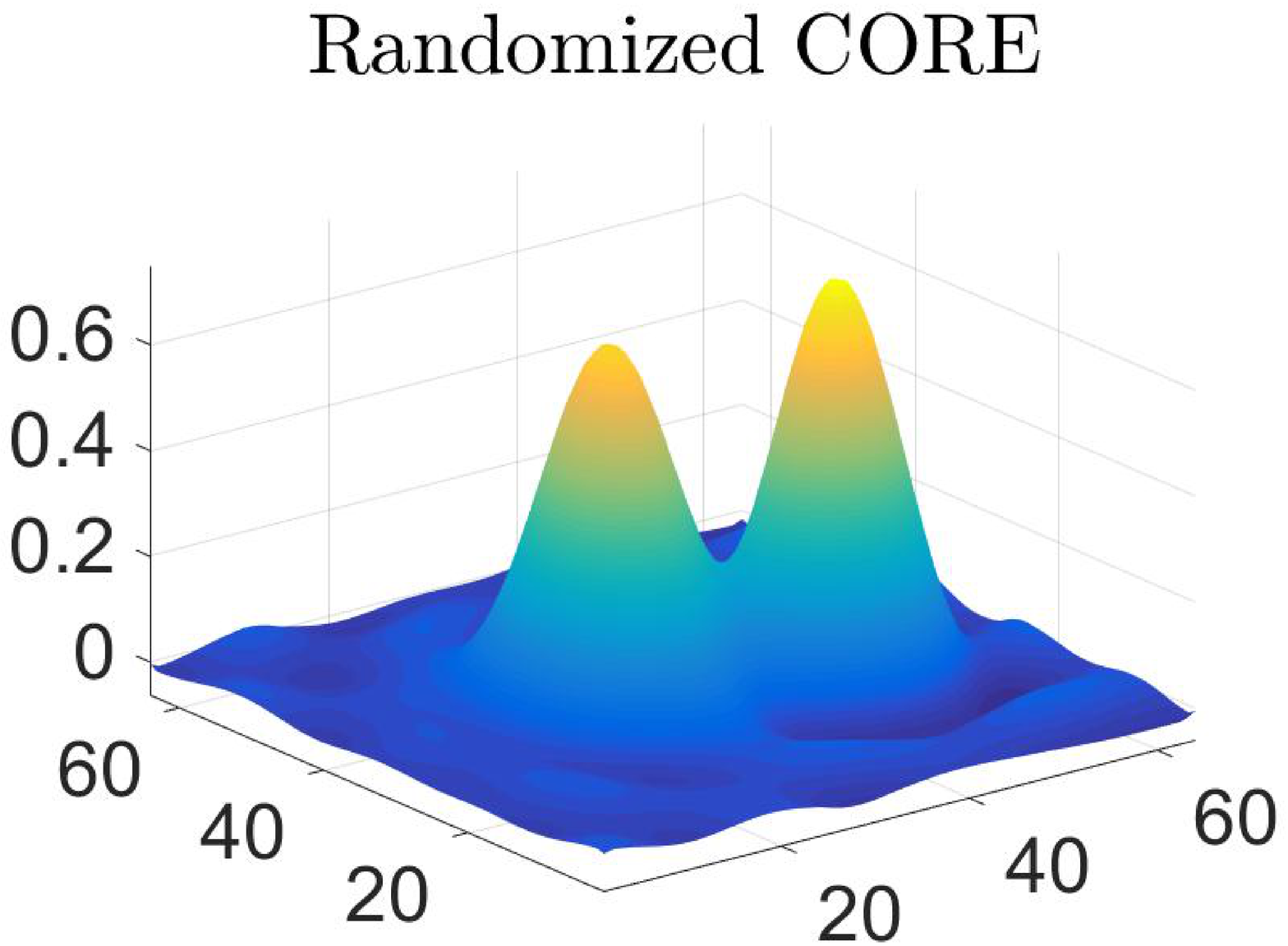}}
\subfigure[][$\varepsilon = 0.001$]
    {\includegraphics[scale=0.15]{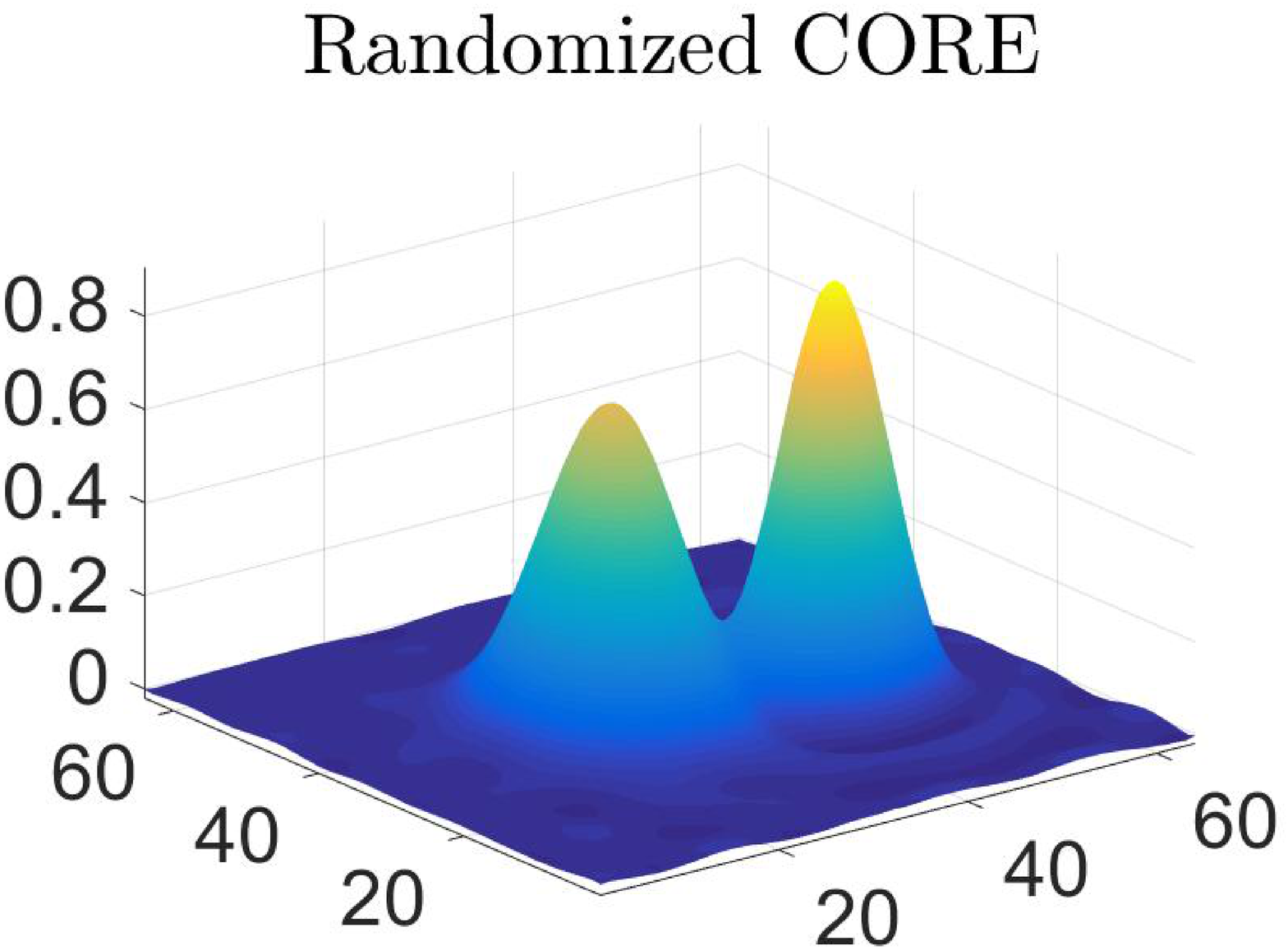}}
\caption{\textsf{PRdiffusion(n)}: Solutions obtained by Algorithm \ref{algo_tls2}  for $n=64$.}\label{fig_ir1}
\end{figure}

\begin{figure}[!ht]\centering
\subfigure[][\textsf{Exact}]
    {\includegraphics[scale=0.15]{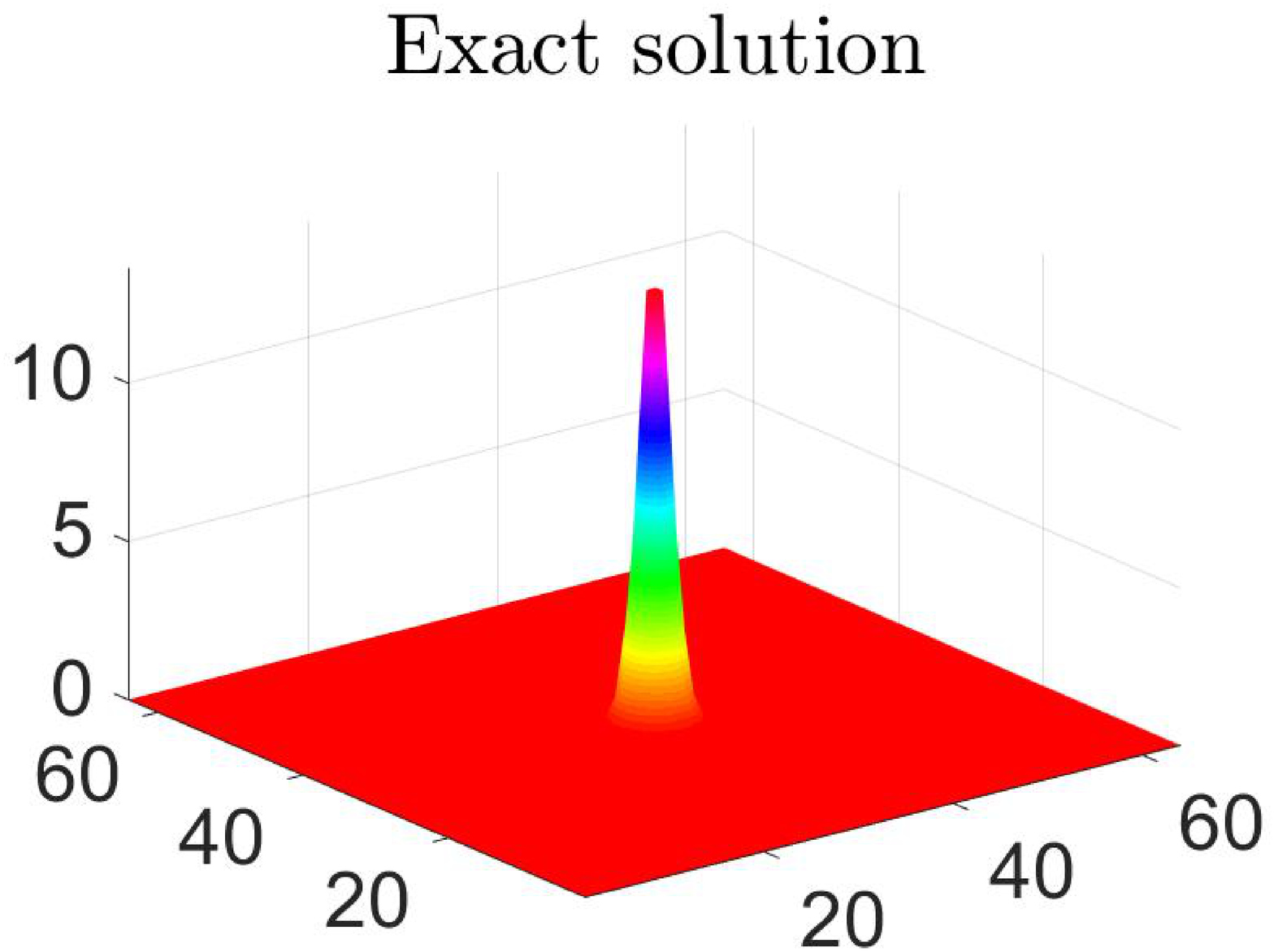}}
\subfigure[][$\varepsilon = 0.1$]
    {\includegraphics[scale=0.15]{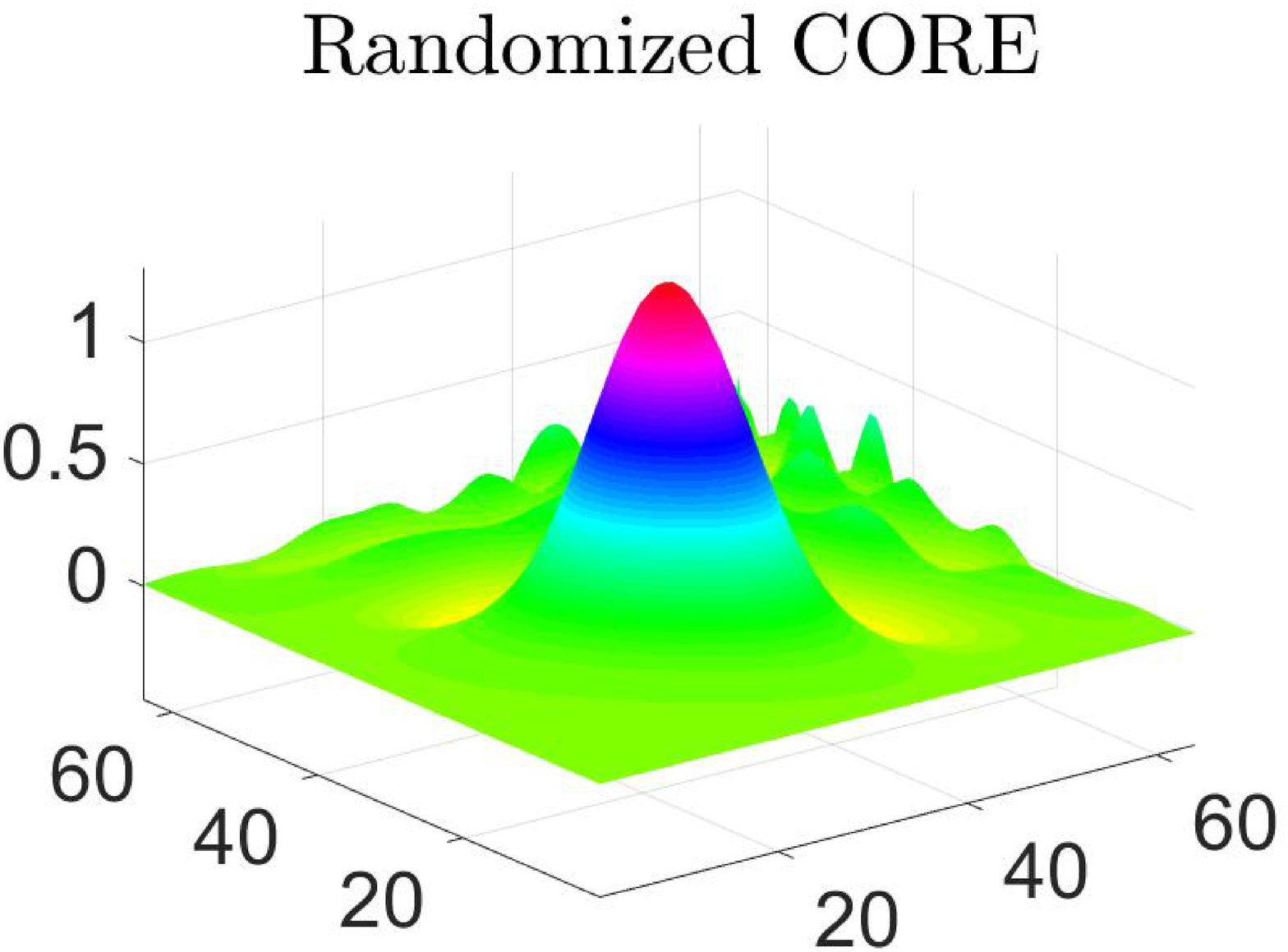}}
\subfigure[][$\varepsilon = 10^{-5}$]
    {\includegraphics[scale=0.15]{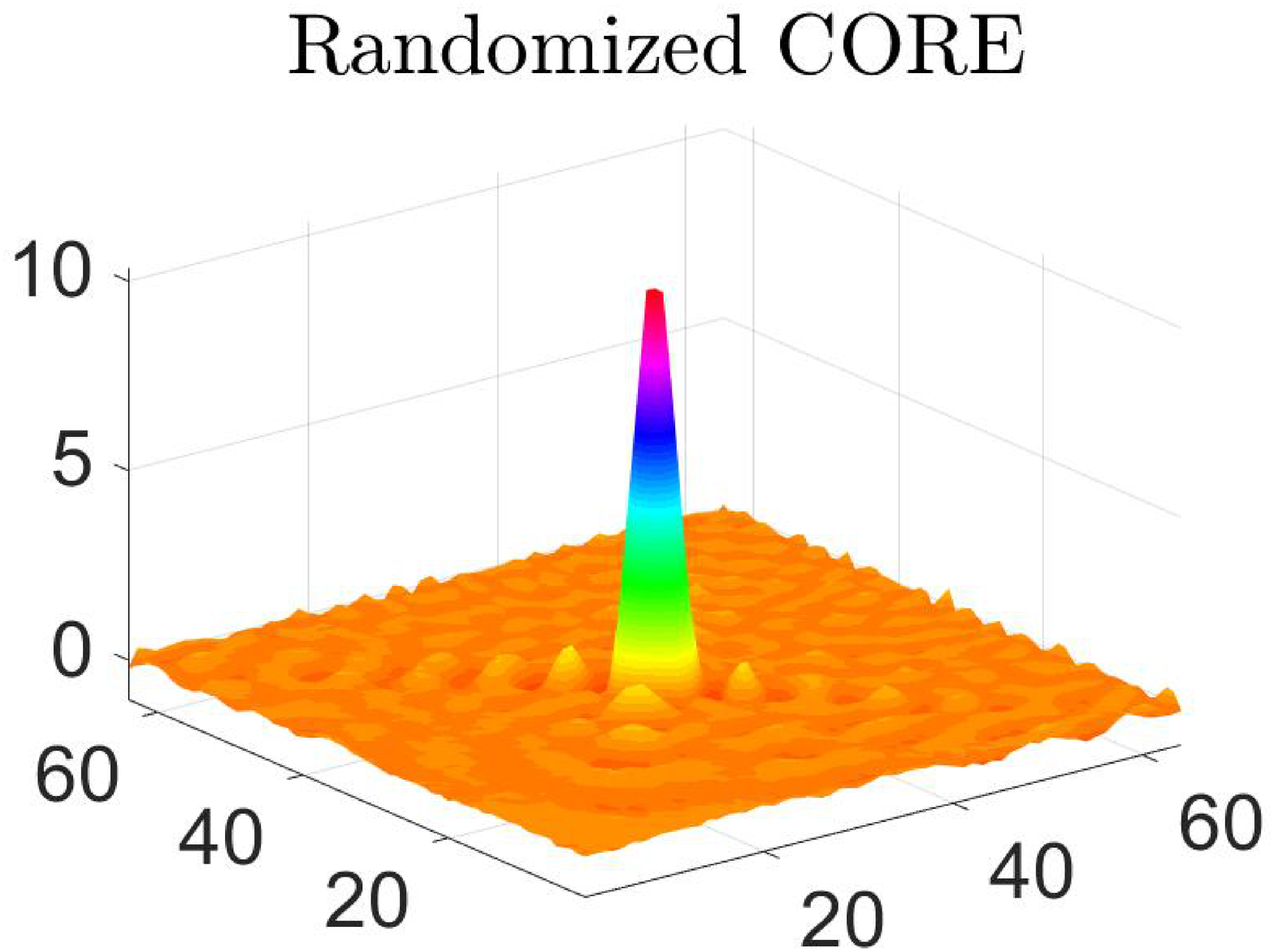}}
\caption{\textsf{PRnmr(n)}: Solutions obtained by Algorithm \ref{algo_tls2} for $n=64$.}\label{fig_ir2}
\end{figure}

\begin{figure}[!ht]\centering
\subfigure[][\textsf{Exact}]
    {\includegraphics[scale=0.12]{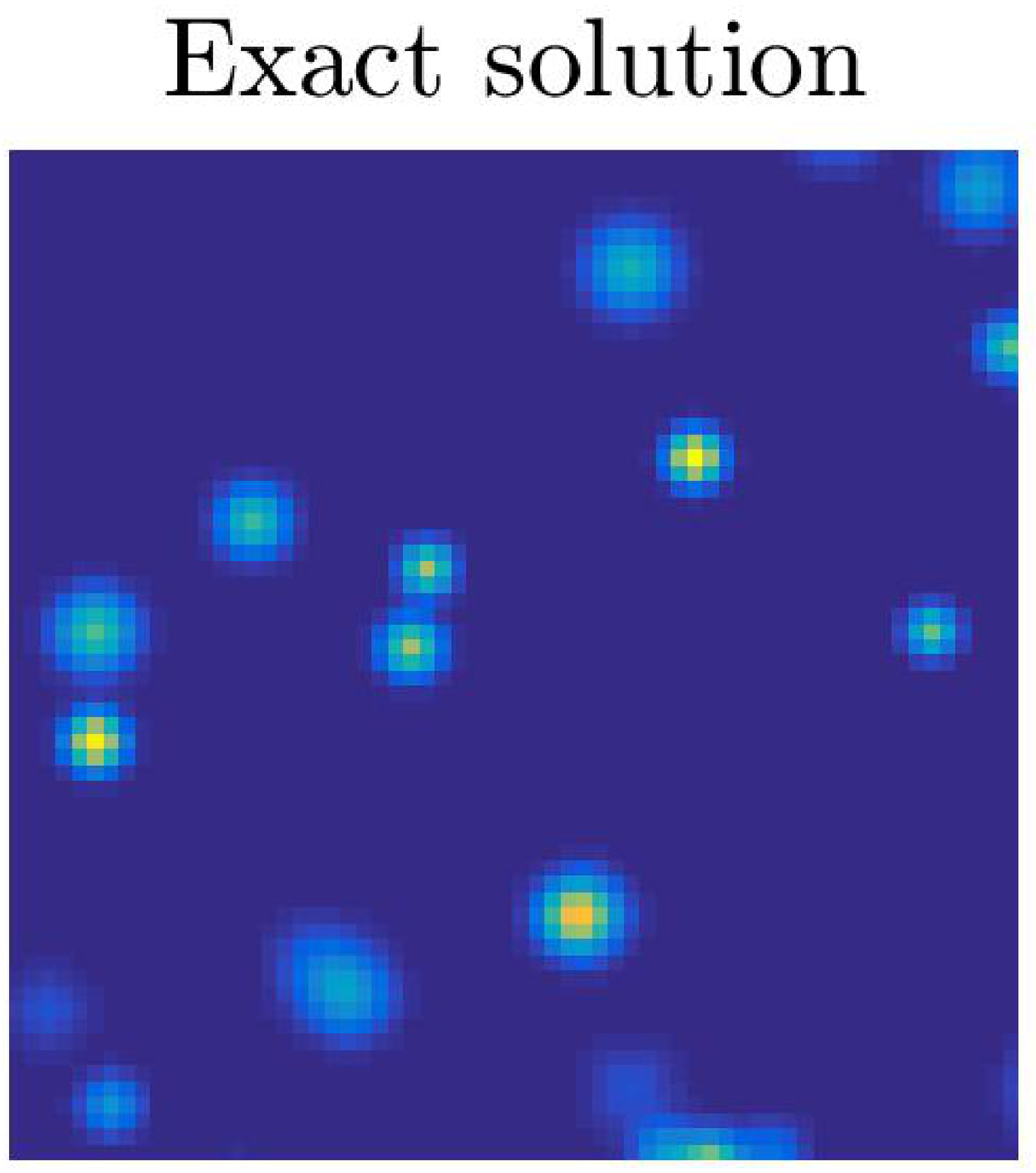}}
\subfigure[][\textsf{Blurred}]
    {\includegraphics[scale=0.12]{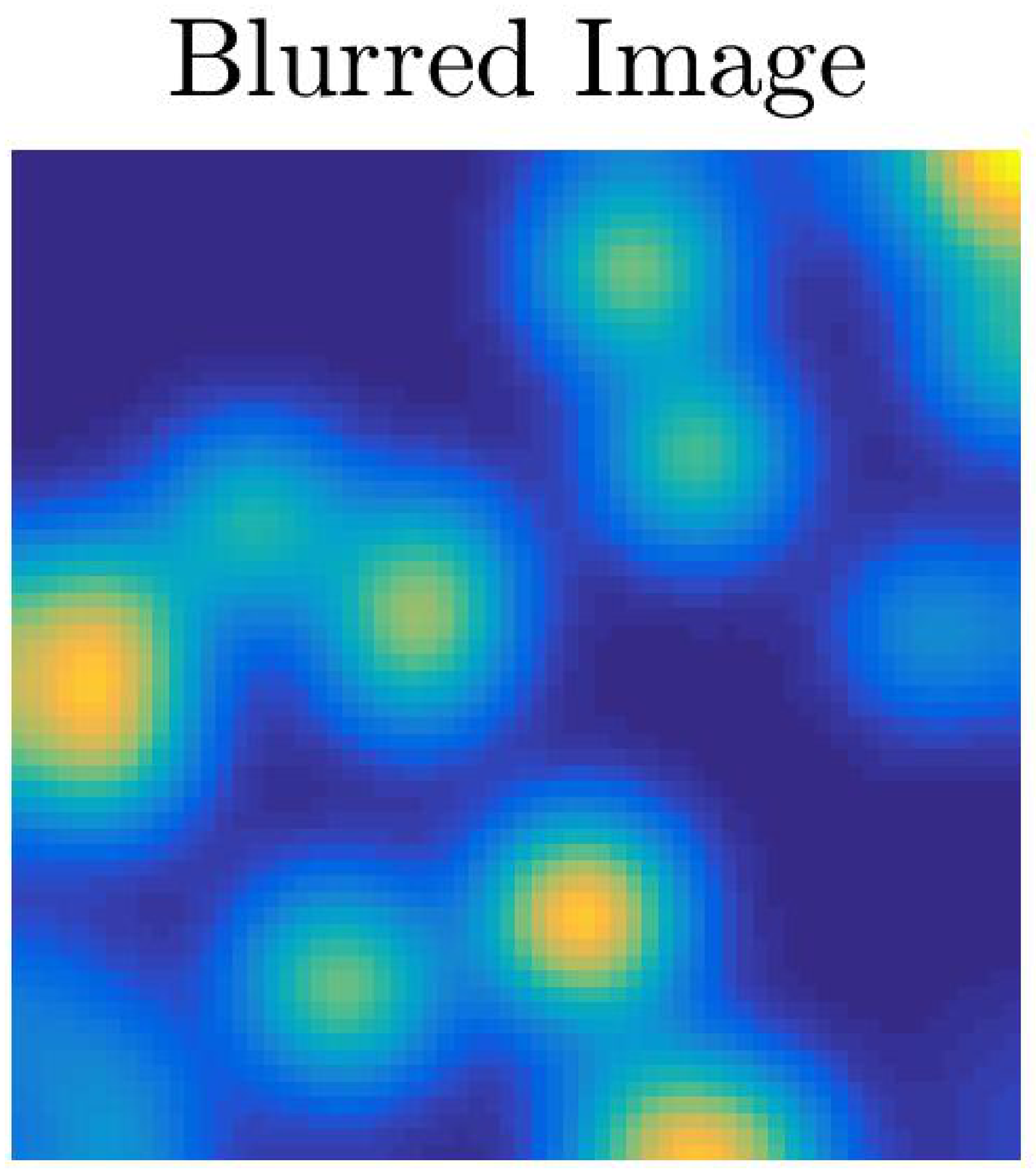}}
\subfigure[][$\varepsilon = 0.1$]
    {\includegraphics[scale=0.12]{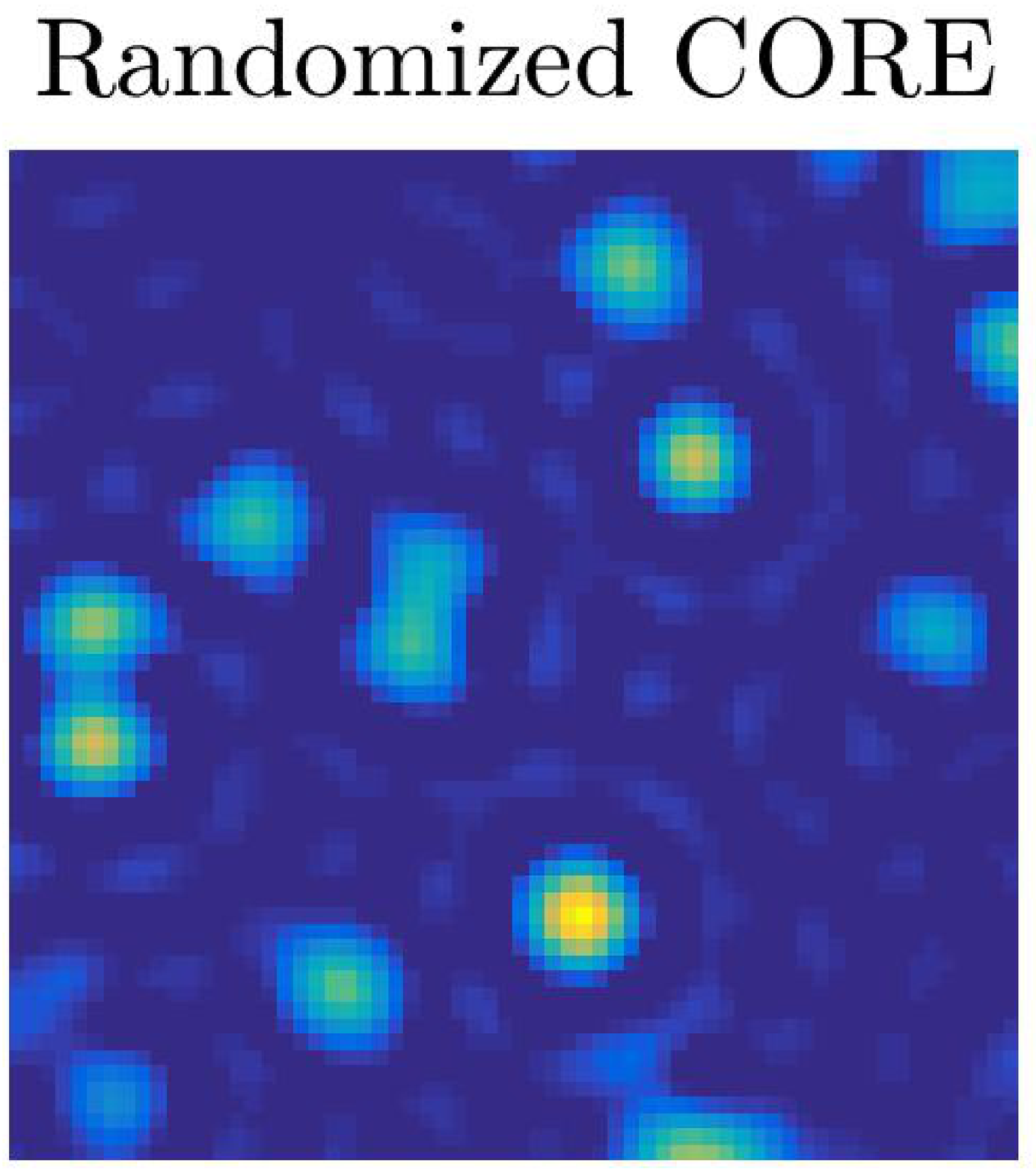}}
\subfigure[][$\varepsilon = 0.001$]
    {\includegraphics[scale=0.12]{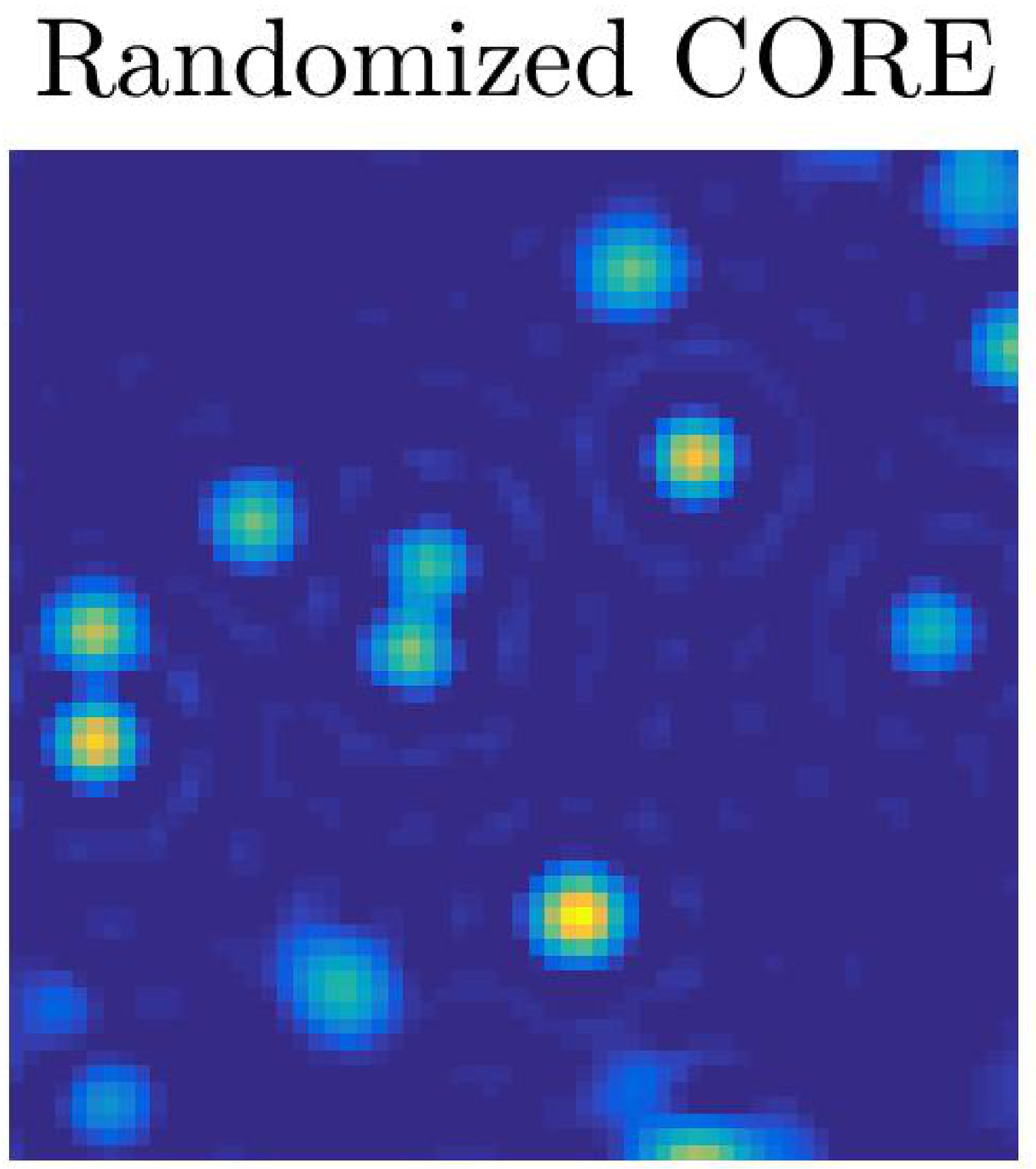}}
\caption{\textsf{PRblurgauss}: Solutions obtained by Algorithm \ref{algo_tls2} for $n=64$.}\label{fig_ir3}
\end{figure}
\section{Conclusion}\label{sec conclusion}
We propose an approximate core reduction and obtain the approximate TLS solution by a randomized algorithm. It can be treated as a regularization technique with the tolerance as a regularization parameter for the ill-posed problem. In theory and numerical experiments, we show that the randomized core reduction is competitive with the truncated TLS in accuracy and more efficient in time and storage. For the large-scale problem, the coefficient matrix does not need to be explicit.
In future, we shall consider the randomized core reduction with multiple right-hand sides and the structured randomized algorithm for the ill-posed problems arising from image restoration and signal processing.

\section*{Acknowledgments}
We wish to thank Prof. Eric King-Wah Chu and Dr. Min Wang who provided useful suggestions for improving the manuscript.
The first author is supported by the National Natural Science Foundation,
People's Republic of China (Grant No. 11601484).
The second author is supported by the International Cooperation Project of Shanghai Municipal
Science and Technology Commission (Grant No. 16510711200) and the National Natural Science Foundation,
People's Republic of China (Grant No. 11771099).

\end{document}